\documentclass[11pt,reqno]{amsart}

\usepackage{mathrsfs}

\usepackage{hyperref}

\usepackage{graphicx}
\usepackage{xcolor}
\usepackage{enumerate}
\usepackage{bm}

\newtheorem{theorem}{Theorem}[section]
\newtheorem{lemma}[theorem]{Lemma}

\newtheorem{corollary}[theorem]{Corollary}

\theoremstyle{definition}
\newtheorem{definition}[theorem]{Definition}

\theoremstyle{remark}
\newtheorem{remark}[theorem]{Remark}

\numberwithin{equation}{section}

\title[The weak solutions to complex Hessian equations  ]{The weak solutions to complex Hessian equations}
\thanks{}

\author{Wei Sun}

\address{Institute of Mathematical Sciences, ShanghaiTech University, Shanghai, China}
\email{sunwei@shanghaitech.edu.cn}

\begin{document}
\setlength{\baselineskip}{1.2\baselineskip}

\begin{abstract}

In this paper, we shall study existence of weak solutions to complex Hessian equations. With appropriate assumptions, it is possible to obtain  weak solutions in pluripotential sense.

\end{abstract}

%\subjclass[2010]{Primary 53C21; Secondary 35J65, 58J32}

\maketitle

%\tableofcontents

\section {Introduction}

Let $(M,\omega)$ be a compact K\"ahler manifold without boundary of complex dimension $n \geq 2$. We are concerned with complex Hessian equation:
\begin{equation}
\label{equation-hessian}
	(\chi + \sqrt{-1} \partial\bar\partial\varphi)^m\wedge\omega^{n - m} = e^{mf} \omega^n ,\quad ess \sup_M \varphi = 0
\end{equation} 
where $2\leq m \leq n$ and $\chi \in \bar\Gamma^m_\omega$. 
In this paper, we say $\chi \in \bar\Gamma^m_\omega$ if $\chi$ is a real $(1,1)$-form and its eigenvalue set with respect to $\omega$ is contained in $\Gamma^m \subset\mathbb{R}^n$. 
Complex Hessian equation can be viewed as a direct extension of complex Monge-Amp\`ere equation.

In domains of $\mathbb{C}^n$, Li~\cite{Li2004} studied the Dirichlet problem with smooth data. Later, Blocki~\cite{Blocki2005} began to establish potential theory for complex Hessian equation on $\mathbb{C}^n$, and described maximal $m$-subharmonic function. 
Dinew and Kolodziej~\cite{DinewKolodziej2014} further developed the potential theory and obtained a priori estimates for weak solutions in pluripotential sense.

On complex manifolds, complex Hessian equation turns to be more difficult than complex Monge-Amp\`ere equation, although the latter seems more non-linear. Hou~\cite{Hou2009} solved Equation~\ref{equation-hessian} with smooth data under the assumptions that $\chi = \omega$ and $\omega$ has nonnegative holomorphic bisectional curvature (see also \cite{Jbilou2010}). Hou, Ma and Wu~\cite{HouMaWu2010} removed the curvature assumption and proved a sharp $C^2$ estimate. Dinew and Kolodziej~\cite{DinewKolodziej2017} obtained a Liouville theorem for complex Hessian equation globally defined on $\mathbb{C}^n$, and then derived a gradient estimate from the sharp $C^2$ estimate of Hou-Ma-Wu by blow-up argument. Later, the author~\cite{Sun2017u}, Zhang~\cite{Zhang2017} and Sz\'ekelyhidi~\cite{Szekelyhidi} extended the results to more general assumptions, e.g. $\omega$ is Hermitian, or $\chi \neq \omega$.

In this paper, we are concerned with weak solutions to Equation~\eqref{equation-hessian} when the given data are not regular enough.  
The difficulty is that the positivity of $\chi + \sqrt{-1} \partial\bar\partial \varphi \in \bar\Gamma^m_\omega$ is pointwisely dependent on $\omega$, which is different than complex Monge-Amp\`ere equation. 
Without the extra assumption of nonnegative holomorphic bisectional curvature, we cannot localize the equation to a domain in $\mathbb{C}^n$, and a lot of techniques cannot be applied, e.g., regularization and Hartogs Lemma. 
%The definition of $m$-$\omega$-subharmonic function is provided in \cite{DinewKolodziej2014}.
A $C^2$-smooth function is called $\varphi$ $m$-$\omega$-subharmonic if 
\begin{equation*}
	(\sqrt{-1} \partial\bar\partial \varphi)^k \wedge \omega^{n - k} \geq 0 , \qquad \forall k = 1, \cdots , m .
\end{equation*}
Then we can define $m$-$\omega$-subharmonic function in current sense~\cite{DinewKolodziej2014} in a formally strong way.
\begin{definition}

In a neighborhood of any given point, there exists a decreasing to $\varphi$  sequence of $C^2$-smooth $m-\omega-sh$ functions $\varphi_j$ such that  for any set of $C^2$ smooth $m-\omega-sh$ functions $v_1, \cdots , v_{m-1}$ the inequality
\begin{equation*}
\sqrt{-1} \partial\bar\partial \varphi_j \wedge \sqrt{-1} \partial\bar\partial v_1 \wedge \cdots \wedge \sqrt{-1} \partial\bar\partial v_{m - 1} \wedge \omega^{n - m} \geq 0
\end{equation*}
is satisfied. Then we call $\varphi$ $m$-$\omega$-subharmonic function.

\end{definition}

Since $\chi$ is closed and hence has a local potential anywhere on $M$, we can simply adapt the definition for complex Hessian equation~\eqref{equation-hessian}.
\begin{definition}

In a neighborhood of any given point, there exists a decreasing to $\varphi$  sequence of $C^2$-smooth functions $\varphi_j$ such that  for any set of $C^2$ smooth $m-\omega-sh$ functions $v_1, \cdots , v_{m-1}$ the inequality
\begin{equation*}
\left(\chi + \sqrt{-1} \partial\bar\partial \varphi_j \right) \wedge \sqrt{-1} \partial\bar\partial v_1 \wedge \cdots \wedge \sqrt{-1} \partial\bar\partial v_{m - 1} \wedge \omega^{n - m} \geq 0
\end{equation*}
is satisfied. Then we call $\varphi$ $m$-$(\chi,\omega)$-subharmonic function.
\end{definition}

%\begin{remark}
For complex Hessian equation on K\"ahler manifolds, it is possible to derive a definition of $m$-pluripotential solution as in \cite{Blocki2005}\cite{DinewKolodziej2014}.  
 %By the $L^\infty$ estimate in \cite{SuiSun2021}, it is possible to find a bounded $m$-pluripotential solution when the right-sided is in $L^p$ $(p > \frac{n^2}{m})$.
%\end{remark}
%
\begin{definition}
Function $\varphi \in L^\infty$ is an admissible solution to complex Hessian equation~\eqref{equation-hessian}, 
if it holds true that $\varphi$ is $m$-$(\chi,\omega)$-subharmonic 
and in pluripotential sense
\begin{equation*}
	(\chi + \sqrt{-1} \partial\bar\partial \varphi)^m \wedge \omega^{n - m} = e^{mf} \omega^n .
\end{equation*}

\end{definition}

The main result of this paper is the following theorem.
\begin{theorem}
\label{main-theorem}

Let $\chi \in \bar\Gamma^m_\omega$ be a closed real $(1,1)$-form, and $\tilde \chi$ be a semipositive and big closed $(1,1)$-form on compact K\"ahler manifolds $(M,\omega)$ without boundary. 
Then there is a unique bounded admissible solution in pluripotential sense to 
\begin{equation}
\label{equation-1}
	(\chi + \tilde \chi + \sqrt{-1} \partial\bar\partial \varphi)^m \wedge \omega^{n - m} = e^{m f} \omega^n , \qquad \sup_M \varphi = 0 ,
\end{equation}
satisfying 
\begin{equation*}
\int_M e^{mf} \omega^n = \int_M (\chi + \tilde \chi)^m \wedge \omega^{n - m} .
\end{equation*}
for $q > 1$ and $e^{nf} \in L^q$. 
In particular, $\varphi$ is smooth in the ample locus of $\tilde \chi$ given that the holomorphic bisectional curvature of $(M,\omega)$ is nonnegative and $e^f$ is smooth.

\end{theorem}
%\begin{remark}
%	The extra assumption is somehow necessary,
%	but I am not sure if it is optimal.
%	An easy example is that if $\chi = 0$ in complex Hessian equation~\eqref{equation-hessian}, then obviously there is no solution.
%\end{remark}

The main technique is a modification of Guo-Phong-Tong~\cite{GPT2021}'s argument, which combines the methods of Wang-Wang-Zhou~\cite{WangWangZhou2021} and Chen-Cheng~\cite{Chen-Cheng}. One of the improvement in our argument is that there is no need to require $\chi$ to be nonnegative in some sense. Moreover, the modification makes it possible to adapt Guo-Phong-Tong~\cite{GPT2021}'s argument to more general equations~\cite{SuiSun2021}. 
%We shall give a brief survey in a later paper~\cite{Sun202304}.

We shall also study viscosity solutions in this paper. There is no difficulty in definition, and the approach is the same as that for pluripotential solutions. In fact, we shall obtain the weak solution by constructing a decreasing function sequence, and then identify whether the solution is pluripotential or viscous.

This paper is organized as follows. In Section 2, we shall briefly recall some notations and formulae which will be in need later on. In Section 3, we shall work on the non-degenerate case of complex Hessian equations, and construct the weak solutions in viscous sense or pluripotential sense. In Section 4, we shall work on the degenerate case of complex Hessian equations, and construct the weak solutions in pluripotential sense. It is worth a mention that the degenerate case means that a background metric is {\em degenerate}. In fact, the corresponding complex Monge-Amp\`ere equation is said to be {\em singular}. In Section 5, the uniqueness will be proven for complex Hessian equations on compact K\"ahler manifolds in current sense. In Section 6, we shall discover the regularity of weak solutions when holomorphic bisectional curvature is nonnegative.

\medskip
\section{Preliminaries}

In this section, we shall state some elementary notations  and formulae.

\medskip
\subsection{Elementary symmetric polynomial}

The $m$-th elementary symmetric polynomial is 
\begin{equation}
	S_m (\bm{\lambda}) := \sum_{0 < j_1 < j_2 < \cdots < j_m \leq 0} \lambda_{j_1} \lambda_{j_2} \cdots \lambda_{j_m}
\end{equation}
where $\bm{\lambda} = (\lambda_1, \cdots ,\lambda_n) \in \mathbb{R}^n$. The $m$-convex cone is defined by
\begin{equation}
	\Gamma^m = \{\bm{\lambda} \in \mathbb{R}^n | S_1 (\bm{\lambda}) > 0, S_2 (\bm{\lambda}) > 0 , \cdots , S_m (\bm{\lambda}) > 0\} .
\end{equation}
It is well known that $S_m$ has the following properties. 
\begin{theorem}[Maclaurin's inequality]
If $\bm{\lambda} \in \Gamma^m$, then 
\begin{equation}
	\left( \frac{S_j (\bm{\lambda})}{C^j_n}\right)^{\frac{1}{j}} 
	\geq
	\left( \frac{S_i (\bm{\lambda})}{C^i_n}\right)^{\frac{1}{i}}  
	,
	\qquad
	\text{for all } 1 \leq j \leq i \leq m ,
\end{equation}
where
$C^i_n := \frac{n!}{i! (n - i)!}$.
\end{theorem}

\begin{theorem}[G\r{a}rding's inequality]
For any $\bm{\lambda}, \bm{\eta} \in \Gamma^m$,
\begin{equation}
	\sum^n_{i = 1} \eta_i S_{m - 1;i} (\bm{\lambda}) \geq m S^{\frac{1}{m}}_m (\bm{\eta}) S^{\frac{m - 1}{m}}_m (\bm{\lambda}) ,
\end{equation}
where $S_{m-1;i} (\bm{\lambda}) := S_{m - 1} ((\lambda_1, \cdots , \lambda_{i - 1}, 0 , \lambda_{i+1}, \cdots , \lambda_m))= \frac{\partial S_m}{\partial \lambda_i}(\bm{\lambda})$.
In particular, $S^{\frac{1}{m}}_m(\bm{\lambda})$ is concave in $\Gamma^m$ and homogeneous of degree $1$.

\end{theorem}

Let $A$ be a $n\times n$ Hermitian matrix, and $\bm{\lambda} (A)$ be the eigenvalue set of $A$. 
We denote
\begin{equation}
\label{definition-2-5}
	S_m (A) := S_m (\bm{\lambda} (A)) = \sum_{\text{principal minor } A' \text{ of order } m} \det A' .
\end{equation} 
Further, let $\chi$ be a real $(1,1)$-form on $M$, and $\bm{\lambda} (\chi)$ be the eigenvalue set of $\chi$ with respect to $\omega$ under any local chart. We denote
\begin{equation}
\label{definition-2-6}
	S_m (\chi) := S_m (\bm{\lambda} (\chi)) ,
\end{equation}
 and
 \begin{equation}
 	\Gamma^m_\omega := \{\text{real } (1,1)-\text{form } \chi | \bm{\lambda} (\chi) \in \Gamma^m\} .
 \end{equation}
At a fixed point $\mathcal{Z} \in M$, we can choose a local chart around $\mathcal{Z}$ such that $\omega_{i\bar j} (\mathcal{Z}) = \delta_{ij}$, and still denote by $\chi$ its corresponding Hermitian coefficient matrix under this chart. Then we can see that \eqref{definition-2-5} and \eqref{definition-2-6} are compliant. Moreover, it is easy to see that
\begin{equation*}
	S_m (\chi) = C^m_n \frac{\chi^m \wedge \omega^{n - m}}{\omega^n} .
\end{equation*}
%where $C^m_n := \frac{n!}{m! (n - m)!}$.

\medskip
\subsection{Iteration argument}

The key tools for the estimates are iteration arguments. The following lemma~\cite{GP2022} is derived by an iteration argument in \cite{Kolodziej}. 
 \begin{lemma}
 \label{lemma-3-2}
 Let $\phi (s)$ be a monotone increasing function on $(0,s_0)$ satisfying that
 \begin{equation*}
 	\lim_{s \to 0+} \phi (s) = 0,
 \end{equation*}
 \begin{equation*}
 	\phi (s) > 0 \quad \text{for } s \in (0,s_0) ,
 \end{equation*}
 and
 \begin{equation*}
 	t \phi(s - t) \leq C_0 (\phi (s))^{1 + \delta_0} \quad \text{for all } 0 < t < s < s_0 .
 \end{equation*}
Then we have
\begin{equation}
	\phi (s_0) \geq \left( \frac{s_0 (1 - 2^{-\delta_0})}{2 C_0}\right)^{\frac{1}{\delta_0}} .
\end{equation}

 \end{lemma}

The following is De Giorgi iteration. For a proof, we refer the readers to \cite{ChenWu}\cite{GilbargTrudinger}. 
\begin{lemma}[De Giorgi iteration]
\label{lemma-4-1}
Suppose that $\phi (s)$ is a nonnegative increasing function on $[s_0,+\infty)$ satisfying
\begin{equation*}
	s'^{\alpha} \phi (s' + s) \leq C \phi^{1 + \delta} (s), \qquad \forall s' >0, s\geq s_0
\end{equation*}
where $\alpha, \delta > 0$. Then $\phi (s_0 + d) = 0$, for all $d \geq C^{\frac{1}{\alpha}} \phi^{\frac{\delta}{\alpha}} (s_0) 2^{\frac{1 + \delta}{\delta}}$ .
\end{lemma}

\medskip
\section{Non-degenerate Case}
 
In this section, we shall study a class of non-degenerate complex equations,
\begin{equation}
\label{equation-3-1}
	F (\chi + \delta \omega + \sqrt{-1} \partial\bar\partial \varphi) 
	:= \mathcal{F} (\bm{\lambda} (\chi + \delta \omega + \sqrt{-1} \partial\bar\partial\varphi)) 
	=  \psi, \quad \sup_M \varphi = 0,
\end{equation}
where $\chi$ is a real $(1,1)$-form and $\bm{\lambda} (\chi)$ denotes the eigenvalue set of $\chi$ with respect to $\omega$. The smooth function $\mathcal{F}$ is defined on a symmetric convex cone $\Gamma \in \mathbb{R}^n$ with $\Gamma^n \subset \Gamma \subset \Gamma^1$.

% non-degenerate case of complex Hessian equations, 
%which means that
%there exists $\delta > 0$ such that  $\chi - \delta \omega \in \bar\Gamma_\omega$ for $\chi$ in Equation~\eqref{equation-hessian}.

\subsection{$L^\infty$ estimate} 
\label{subsection-3-1}
%For generality and convenience, we consider the following equation which includes the non-degenerate case of Hessian equations~\eqref{equation-hessian}:
% \begin{equation}
% \label{equation-2}
% 	 	F(\chi + \delta\omega + \sqrt{-1} \partial\bar\partial \varphi) =  \psi , \quad \sup_M \varphi = 0 .
% \end{equation}
We impose the following structure condition for Equation~\eqref{equation-3-1} as in \cite{SuiSun2021}:
% \begin{enumerate}
 %
%   \item  $\chi $ is a  real $(1, 1)$-form with $\bm{\lambda} (\chi ) \in {\bar\Gamma}$;
%   , and its eigenvalue set  with respect to $\omega$ belongs to a symmetric cone $\bar \Gamma \subset \mathbb{R}^n$ with $\Gamma^1 \subset \Gamma \subset\Gamma^n$
   %
%
   %
%   \item 
   there is a smooth function $f$ on $M$ such that 
   \begin{equation}
   \label{structure-condition-1}
 	\hat\chi^n (z) \leq e^{n f(z)} \omega^n (z) ,
   \end{equation}
   if at point $z \in M$, real $(1,1)$-form $\hat\chi (z) \geq 0$ satisfies
    \begin{equation*}
        F\left(\chi(z) + \hat \chi(z)\right)  \leq \psi (z) .
     \end{equation*}
   %
% \end{enumerate}
As pointed out in \cite{SuiSun2021}, the structure condition is a direct extension of $\mathcal{C}$-subsolution~\cite{Szekelyhidi} (or cone condition~\cite{SongWeinkove,FangLaiMa,Sun2017}). 
%We shall find the $L^\infty$ estimate for an admissible solution $\varphi$.
 %, that is,   $\chi +\delta\omega + i\partial\bar\partial\varphi \in \bar\Gamma_\omega$.
%
% In this section, we shall consider the $L^\infty$ estimate for the following equations
% \begin{equation}
% 	F(\chi + \delta\omega + i\partial\bar\partial \varphi) =  \psi , \quad \sup_M \varphi = 0 .
% \end{equation}
%The equation satisfies the following condition:
%there is a function $f$ on $M$ such that for any $\hat\chi (z) \geq 0$ satisfying
% $$
% F(\chi (z) + \hat\chi (z)) \leq \psi (z),
% $$
% it must be 
% \begin{equation}
% \hat \chi^n (z) \leq e^{n f (z)} \omega^n (z) .
% \end{equation}
If $\chi \in \Gamma^m_\omega$ in complex Hessian equation~\eqref{equation-hessian}, i.e. $\bm{\lambda}(\chi) \in \Gamma^m$, 
then there is  a small constant $\delta  > 0$ so that $\chi - \delta \omega \in \Gamma^m_\omega$, which we call the non-degenerate case of complex Hessian equation. The non-degenerate complex Hessian equation is a particular case of Equation~\eqref{equation-3-1} if replacing $\chi$ by $\chi - \delta \omega$ in \eqref{equation-3-1}.

We can prove $L^\infty$ estimate for solution $\varphi$, even when $\chi$ and $\omega$ are not closed,
%i.e. $\bm{\lambda} (\chi + \delta \omega + \sqrt{-1} \partial\bar\partial \varphi) \in \Gamma$. 
%The closedness of $\chi$ and $\omega$ is actually not in need in this section, 
by adapting the argument in Guo-Phong~\cite{GP2022}.
\begin{theorem}
\label{theorem-3-1}
Besides structure condition~\eqref{structure-condition-1}  above, we also suppose that $\int_M e^{nf} (1 + n |f|)^p$ is bounded for $p > n $. 
%With the assumptions above,  
Then there is a constant $C$ such that $
- \varphi < C $,
for solution $\varphi$ to Equation~\eqref{equation-3-1}. 
%if $||e^{nf}||_{L^1 (\log L)^p}:= \int_M e^{nf} (1 + n |f|)^p$ is bounded. 
\end{theorem}

 \begin{proof}%[Proof of Theorem~\ref{theorem-3-1}]
Due to compactness of $(M,\omega)$, there exists $r_0 > 0$  such that for all point $\mathcal{Z} \in M$, there exists a local chart satisfying $\omega_{i\bar j} (\mathcal{Z}) = \delta_{ij}$ and 
 \begin{equation*}
 \frac{1}{2} \sqrt{-1}  \sum_{i,j} \delta_{ij} dz^i \wedge d\bar z^j \leq \omega \leq 2 \sqrt{-1} \sum_{i,j} \delta_{ij} dz^i \wedge d\bar z^j \quad \text{ in } B(\mathcal{Z},2r_0),
 \end{equation*}
 where $B(\mathcal{Z},2r_0)$ is an Euclidian ball in $\mathbb{C}^n$.
%Moreover, the coordinate of $\mathcal{Z}$ is the origin. 
Now we choose  the minimum point of $\varphi$ on $M$ as point $\mathcal{Z}$. 
 Let
 \begin{equation*}
 u_s (z) = \varphi (z) - \varphi(\mathcal{Z}) + \frac{\delta}{2} |z|^2 - s ,
 \end{equation*}
 where $0 < s < s_0 := 2 \delta r^2_0$ and $\mathcal{Z}$ is the origin in the local chart. 
 We denote
 \begin{equation*}
 	\Omega := B (\mathcal{Z},2r_0),
 \end{equation*}
 and
 \begin{equation*}   
 	\Omega_s := \{ u_s(z) < 0\} \subset \mathring{\Omega}.
 \end{equation*}
% On $\partial \Omega$, 
% $$
% u_s (z) =  \varphi (z) - \varphi(\mathcal{Z}) + 2 \delta r^2_0 - s > 0 .
% $$
% Let 
% $$
% \Omega_s = \{ u_s(z) < 0\}.
% $$

 We solve the Dirichlet problem
 \begin{equation}
 \left\{
 \begin{aligned}
 \left(\sqrt{-1} \partial \bar\partial \psi_{s,k}\right)^n &= \frac{\tau_k (- u_s)}{A_{s,k}} e^{nf} \omega^n , &&\text{ in } \Omega ,\\
 \psi_{s,k} &= 0, && \text{ on } \partial\Omega ,
 \end{aligned}
 \right.
 \end{equation}
 where
 \begin{equation*}
  A_{s,k} := \int_\Omega \tau_k (- u_s) e^{nf} \omega^n \to A_s := \int_{\Omega_s} (- u_s) e^{nf} \omega^n ,
\end{equation*}
  and
 $\tau_k : \mathbb{R} \to \mathbb{R}^+$  is a decreasing sequence of smooth positive functions such that
 \begin{equation*}
 \tau_k (x) =
 \left\{
 \begin{aligned}
  &\;x + \dfrac{3}{k}\;, &&\quad \text{when } x \geq - \dfrac{1}{k} ,\\
  &\;\frac{1}{k}\;,  &&\quad \text{when } x\leq - \frac{2}{k}
   \end{aligned}
 \right.
 \end{equation*}
 and otherwise
 \begin{equation*}
 \frac{1}{k} \leq \tau_k (x) \leq \frac{2}{k} %, \quad \text{when } x \in \left[-\dfrac{2}{k}, - \dfrac{1}{k}\right] 
 .
 \end{equation*}
% So
% $$
% \int_\Omega (i\partial\bar\partial \psi_{s,k})^n = 1
% $$
% and 
% $$
% A_{s,k} \to A_s := \int_{\Omega_s} (- u_s) e^{nf} \omega^n . 
% $$
% 
%
%
Define  
\begin{equation*}
\Phi := - \epsilon (-\psi_{s,k})^{\frac{n}{n+1}} - u_s , 
\end{equation*}
where %$\epsilon$ is to be specified.
% Now we choose $\epsilon$ by
 \begin{equation*}
  \epsilon := \left(\frac{n}{n + 1}\right)^{-\frac{n}{n + 1}}  A^{\frac{1}{n+1}}_{s,k} .
 \end{equation*}
%Since $u_s \geq 0 $ on $\partial \Omega$, 
The function $\Phi$ is continuous, and hence must reach $\sup_{\bar\Omega} \Phi$ at some point $z_{max} \in \bar\Omega$. 
 If $z_{max} \in \bar\Omega \setminus \Omega_s$,
 \begin{equation*}
 \Phi (z_{max}) \leq - \epsilon (-\psi_{s,k} (z_{max}))^{\frac{n}{n + 1}} < 0.
 \end{equation*}
 If $z_{max} \in \Omega_s \subset \mathring{\Omega}$,
 \begin{equation*}
 \begin{aligned}
 	0
% 	&\geq \left[\partial_i \bar\partial_j \Phi\right] \\
% 	&= \left[  \epsilon \frac{n}{n + 1} (-\psi_{s,k})^{-\frac{1}{n + 1}} \partial_i\bar\partial_j \psi_{s,k} + \epsilon \frac{n}{(n + 1)^2} (-\psi_{s,k})^{-\frac{n + 2}{n + 1}} \partial_i\psi_{s,k} \bar\partial_j \psi_{s,k} - \partial_i\bar\partial_j u_s\right] \\
 	&\geq \left[  \epsilon \frac{n}{n + 1} (-\psi_{s,k})^{-\frac{1}{n + 1}} \partial_i\bar\partial_j \psi_{s,k}  - \partial_i\bar\partial_j u_s\right] \\
% 	&= \left[  \epsilon \frac{n}{n + 1} (-\psi_{s,k})^{-\frac{1}{n + 1}} \partial_i\bar\partial_j \psi_{s,k}  - \partial_i\bar\partial_j \varphi - \frac{\delta}{2} \delta_{i\bar j}\right] \\	
 	&\geq \left[  \epsilon \frac{n}{n + 1} (-\psi_{s,k})^{-\frac{1}{n + 1}} \partial_i\bar\partial_j \psi_{s,k}  - \partial_i\bar\partial_j \varphi - \delta \omega_{i\bar j} \right] ,
 \end{aligned}
 \end{equation*}
 and consequently
 \begin{equation*}
 \left[\delta\omega_{i\bar j} + \partial_i\bar\partial_j \varphi\right] 
  \geq \epsilon \frac{n}{n + 1} (-\psi_{s,k})^{-\frac{1}{n + 1}} [\partial_i\bar\partial_j \psi_{s,k}] 
  \geq 0 .
 \end{equation*}
 Therefore,
 \begin{equation}
   1
   \geq
    \epsilon^n \left(\frac{n}{n + 1}\right)^n (-\psi_{s,k})^{-\frac{n}{n + 1}} \frac{\tau_k (-u_s)}{A_{s,k}}  
    \geq \
    \epsilon^n \left(\frac{n}{n + 1}\right)^n (-\psi_{s,k})^{-\frac{n}{n + 1}} \frac{ -u_s}{A_{s,k}} .
 \end{equation}
% \begin{equation}
% \begin{aligned}
% \left[\delta\omega + \partial_i\bar\partial_j \varphi\right] 
% \geq \epsilon \frac{n}{n + 1} (-\psi_{s,k})^{-\frac{1}{n + 1}} [\partial_i\bar\partial_j \psi_{s,k}] \\
% \\
% \\
% \det [\delta\omega + \partial_i\bar\partial_j \varphi] \geq \epsilon^n \left(\frac{n}{n + 1}\right)^n (-\psi_{s,k})^{-\frac{n}{n + 1}} \det [\partial_i\bar\partial_j \psi_{s,k}] \\
% \\
% \\
%  e^{nf} \omega^n \geq \epsilon^n \left(\frac{n}{n + 1}\right)^n (-\psi_{s,k})^{-\frac{n}{n + 1}} \frac{\tau_k (-u_s)}{A_{s,k}} e^{nf} \omega^n \\
%  \\
%  \\
%   1 \geq \epsilon^n \left(\frac{n}{n + 1}\right)^n (-\psi_{s,k})^{-\frac{n}{n + 1}} \frac{\tau_k (-u_s)}{A_{s,k}}  \\
%   \\
%   \\
%  \epsilon^{-n} \left(\frac{n}{n + 1}\right)^{-n} (-\psi_{s,k})^{\frac{n}{n + 1}}  A_{s,k} \geq  -u_s   \\
% \end{aligned}
% \end{equation}
% Now we choose $\epsilon$ by
% \begin{equation}
%  \epsilon = \left(\frac{n}{n + 1}\right)^{-\frac{n}{n + 1}}  A^{\frac{1}{n+1}}_{s,k} ,
% \end{equation}
% $$
% \epsilon^{n + 1} = \left(\frac{n}{n + 1}\right)^{-n}  A_{s,k} ,
% $$
and we can conclude that  $ \Phi \leq 0 $ on $\bar \Omega$.
%In sum, 
%\begin{equation*}
% - u_s \leq \epsilon (-\psi_{s,k})^{\frac{n}{n + 1}} = \left(\frac{n + 1}{n} \right)^{\frac{n}{n + 1}} A^{\frac{1}{n + 1}}_{s,k} (-\psi_{s,k})^{\frac{n}{n + 1}}  \quad \text{ on } \bar\Omega.
% \end{equation*}
% 
% 

According to the works of \cite{Kolodziej}\cite{WangWnagZhou2020}, there are constants $\beta > 0$ and $C > 0$ depending on geometric data such that
% 
% \newpage
% 
% $$
% (- u_s)^{\frac{n + 1}{n}} \leq   \frac{n + 1}{n}  A^{\frac{1}{n }}_{s,k} (-\psi_{s,k})
% $$
% 
% 
% 
% $$
% \frac{(- u_s)^{\frac{n + 1}{n}} }{A^{\frac{1}{n }}_{s,k}} \leq   \frac{n + 1}{n}   (-\psi_{s,k})
% $$
% 
% 
% 
 \begin{equation}
 \label{inequality-3-9}
 \int_{\Omega_s} \exp \left( \beta \frac{(- u_s)^{\frac{n + 1}{n}} }{A^{\frac{1}{n }}_{s,k}}  \right) \omega^n
 % \leq \int_{\Omega_s} \exp \left( \beta \frac{n + 1}{n}   (-\psi_{s,k}) \right) \omega^n 
  \leq \int_{\Omega} \exp \left( \beta \frac{n + 1}{n}   (-\psi_{s,k}) \right) \omega^n 
  \leq C .
 \end{equation}
 Letting $k \to \infty$, we have
 \begin{equation}
 \label{inequality-3-10}
 \int_{\Omega_s} \exp \left( \beta \frac{(- u_s)^{\frac{n + 1}{n}} }{A^{\frac{1}{n }}_{s}}  \right) \omega^n
  \leq C .
 \end{equation}
 %where $C$ depends on $n$ and $diam (\Omega)$.
% 
% $$
% \begin{aligned}
% 	v(z)^p e^{nf} 
% 	&\leq e^{nf} \left(\ln \left(1 + e^{nf}\right)\right)^p + v(z)^p \left(e^{v(z)} - 1\right) \\
% 	&\leq e^{nf} \left(1 + |nf|\right)^p + C_p e^{2 v(z)} .
% \end{aligned}
% $$
% Define
% $$
% v(z) := \frac{\beta}{2} \frac{(- u_s)^{\frac{n + 1}{n}}}{A^{\frac{1}{n}}_s} .
% $$
 By generalized Young's inequality,
\begin{equation}
\label{inequality-3-11}
  \begin{aligned}
  	&\quad \int_{\Omega_s} \left(\frac{\beta}{2}\right)^p \frac{ (- u_s)^{\frac{(n + 1) p}{n}} }{ A^{\frac{p}{n}}_s} e^{nf} \omega^n \\
  	&\leq \int_{\Omega_s} e^{nf} (1 + |nf|)^p \omega^n + C_p \int_{\Omega_s} \exp \left(\beta \frac{(- u_s)^{\frac{n + 1}{n}}}{A^{\frac{1}{n}}_s}\right) \omega^n 
%  	\\
% 	&\leq \int_{\Omega_s} e^{nf} (1 + |nf|)^p \omega^n + C C_p \\
%  	&\leq \int_{\Omega} e^{nf} (1 + |nf|)^p \omega^n + C C_p ,
%  	&\leq C
.
\end{aligned}
\end{equation}
Substituting \eqref{inequality-3-10} into \eqref{inequality-3-11}, 
\begin{equation}
\label{inequality-3-12}
 \int_{\Omega_s} (- u_s)^{\frac{(n + 1) p}{n}} e^{nf} \omega^n 
 	\leq C \left(\frac{2}{\beta}\right)^p A^{\frac{p}{n}}_s  
  	\leq C A^{\frac{p}{n}}_s .
\end{equation}
Applying inequality~\eqref{inequality-3-12} and H\"older inequality with respect to $e^{nf} \omega^n$  to $A_s$,
\begin{equation}
\label{2-13}
\begin{aligned}
  	A_s 
  	&= \int_{\Omega_s} (- u_s) e^{nf} \omega^n \\
  	&\leq \left(\int_{\Omega_s} (- u_s)^{\frac{(n + 1) p}{n}} e^{nf} \omega^n\right)^{\frac{n}{(n + 1)p}} \left(\int_{\Omega_s} e^{nf} \omega^n\right)^{1 - \frac{n}{(n + 1)p}} \\
  	&\leq C A^{\frac{1}{n + 1}}_s \left(\int_{\Omega_s} e^{nf} \omega^n\right)^{1 - \frac{n}{(n + 1)p}}  .
  \end{aligned}
\end{equation}
% $$
% \begin{aligned}
% 	\left(\frac{\beta}{2}\right)^p \frac{ (- u_s)^{\frac{(n + 1) p}{n}} }{ A^{\frac{p}{n}}_s} e^{nf}
% 	&\leq
% 	e^{nf} (1 + |nf|)^p + C_p \exp\left(\beta \frac{(- u_s)^{\frac{n + 1}{n}}}{A^{\frac{1}{n}}_s}\right)
% 	\\
% 	\\
% 	\\
% 	&\quad \int_{\Omega_s} \left(\frac{\beta}{2}\right)^p \frac{ (- u_s)^{\frac{(n + 1) p}{n}} }{ A^{\frac{p}{n}}_s} e^{nf} \omega^n \\
% 	&\leq \int_{\Omega_s} e^{nf} (1 + |nf|)^p \omega^n + C_p \int_{\Omega_s} \exp \left(\beta \frac{(- u_s)^{\frac{n + 1}{n}}}{A^{\frac{1}{n}}_s}\right) \omega^n \\
% 	&\leq \int_{\Omega_s} e^{nf} (1 + |nf|)^p \omega^n + C C_p \\
% 	&\leq \int_{\Omega} e^{nf} (1 + |nf|)^p \omega^n + C C_p \\
% 	&\leq C
% 	\\
% 	\\
% 	\\
% 	&\quad \int_{\Omega_s} (- u_s)^{\frac{(n + 1) p}{n}} e^{nf} \omega^n \\
% 	&\leq \left(\frac{2}{\beta}\right)^p A^{\frac{p}{n}}_s C \\
% 	&\leq C A^{\frac{p}{n}}_s 
% 	\\
% 	\\
% 	\\
% 	A_s 
% 	&= \int_{\Omega_s} (- u_s) e^{nf} \omega^n \\
% 	&\leq \left(\int_{\Omega_s} (- u_s)^{\frac{(n + 1) p}{n}} e^{nf} \omega^n\right)^{\frac{n}{(n + 1)p}} \left(\int_{\Omega_s} e^{nf} \omega^n\right)^{1 - \frac{n}{(n + 1)p}} \\
% 	&\leq C A^{\frac{1}{n + 1}}_s \left(\int_{\Omega_s} e^{nf} \omega^n\right)^{1 - \frac{n}{(n + 1)p}} 
% \end{aligned}
% $$
Rewriting \eqref{2-13}, we have 
 \begin{equation}
 	A_s 
 \leq C \left(\int_{\Omega_s} e^{nf} \omega^n\right)^{1 + \frac{1}{n} - \frac{1}{p}} .
 \end{equation}
% $$
% \begin{aligned}
% 	A^{\frac{n}{n + 1}}_s
% 	&\leq C \left(\int_{\Omega_s} e^{nf} \omega^n\right)^{1 - \frac{n}{(n + 1)p}}
% 	\\
% 	\\
% 	\\
% 	A_s 
% 	&\leq C \left(\int_{\Omega_s} e^{nf} \omega^n\right)^{1 + \frac{1}{n} - \frac{1}{p}}
% \end{aligned}
% $$
% So we need $p > n$.
% 
% 
 Therefore, for $0 < t < s < s_0$,
\begin{equation}
C \left(\int_{\Omega_s} e^{nf} \omega^n\right)^{1 + \frac{1}{n} - \frac{1}{p}} \geq	A_s 
%  	  	&= \int_{\Omega_s} (- u_s) e^{nf} \omega^n \\
  	  	\geq \int_{\Omega_{s - t}} (- u_s) e^{nf} \omega^n 
  	\geq t \int_{\Omega_{s - t}} e^{nf} \omega^n .
\end{equation}
% $$
% \begin{aligned}
% 	A_s
% 	&= \int_{\Omega_s} (- u_s) e^{nf} \omega^n \\
% 	&\geq \int_{\Omega_{s - t}} (- u_s) e^{nf} \omega^n \\
% 	&\geq t \int_{\Omega_{s - t}} e^{nf} \omega^n 
% 	\\
% 	\\
% 	\\
% 	&\quad C \left(\int_{\Omega_s} e^{nf} \omega^n\right)^{1 + \frac{1}{n} - \frac{1}{p}} \\
% 	&\geq t \int_{\Omega_{s - t}} e^{nf} \omega^n 
% \end{aligned}
% $$
% 
%If we define $\phi (s) := \int_{\Omega_s} e^{nf} \omega^n$, 
%then
%\begin{equation}
%\label{inequality-3-16}
% C \phi^{1 + \frac{1}{n} - \frac{1}{p}} (s) \geq t \phi (s - t) .
%\end{equation}
%
%
%By an iteration argument, 
By Lemma~\ref{lemma-3-2}, we can prove that there is a constant $c_0 > 0$ 
%depending on the coefficients $s_0$, $C$ and $\frac{1}{n} - \frac{1}{p}$ in \eqref{inequality-3-16} 
so that 
%\begin{equation}
%\label{inequality-3-17}
%\phi(s_0) \geq c_0 .
%\end{equation}
\begin{equation}
\label{inequality-3-17}
\int_{\Omega_{s_0}} e^{nf} \omega^n
\geq c_0.
\end{equation}

Without loss of generality, we may assume that $\varphi (\mathcal{Z}) < -  s_0 - 1$. 
Otherwise, the proof is finished. 
%Consequently, on $\Omega_{s_0}$,
%% $$
%% u_{s_0} (z) = \varphi (z) - \varphi (\mathcal{Z}) + \frac{\delta}{2} |z|^2 - s_0 < 0,
%% $$
%% so
% \begin{equation*}
% \varphi (z) + \frac{\delta}{2} |z|^2 < \varphi (\mathcal{Z}) + s_0 < - 1 .
%\end{equation*}
%if $\varphi (\mathcal{Z}) < - 2 s_0$. 
% if  $s_0 < \frac{1}{2}$ and $\varphi (\mathcal{Z}) < - 2$. The proof is finished when the assumption fails.
% Then 
% $$
% \ln \frac{-\varphi - \dfrac{\delta}{2} |z|^2}{ \sqrt{- \varphi (\mathcal{Z}) - s_0 }} > \ln \sqrt{-\varphi (\mathcal{Z}) - s_0} > 0 .
% $$
 %$$
 %\begin{aligned}
 %\int_{\Omega_{s_0}} \ln \frac{-\varphi - \dfrac{\delta}{2} |z|^2}{ \sqrt{- \varphi (\mathcal{Z}) - s_0 }} e^nf \omega^n 
 %&> \int_{\Omega_{s_0}} \ln \sqrt{-\varphi (\mathcal{Z}) - s_0}  e^{nf} \omega^n \\
 %&= \phi (s_0) \ln \sqrt{-\varphi (\mathcal{Z}) - s_0}  
 %\end{aligned}
 %$$ 
Then, we calculate
 \begin{equation}
 \label{inequality-3-18}
  \begin{aligned}
c_0    \ln \sqrt{-\varphi (\mathcal{Z}) - s_0} 
  \leq&\, \phi (s_0) \ln \sqrt{-\varphi (\mathcal{Z}) - s_0}  \\
%  =& \int_{\Omega_{s_0}} \ln \sqrt{-\varphi (\mathcal{Z}) - s_0}  e^{nf} \omega^n \\
  <& \int_{\Omega_{s_0}} \ln \frac{-\varphi - \dfrac{\delta}{2} |z|^2}{ \sqrt{- \varphi (\mathcal{Z}) - s_0 }} e^{nf} \omega^n \\
%  \leq& \int_{\Omega_{s_0}} \left(e^{nf} (\ln (1 + e^{nf}))^p + \ln \frac{-\varphi - \dfrac{\delta}{2} |z|^2}{ \sqrt{- \varphi (\mathcal{Z}) - s_0 }} \left(\exp \left(\left(\ln\frac{-\varphi - \dfrac{\delta}{2} |z|^2}{ \sqrt{- \varphi (\mathcal{Z}) - s_0 }}\right)^{\frac{1}{p}} \right)-  1\right) \right) \omega^n \\
%  \leq & \int_{\Omega_{s_0}} \left(e^{nf} (1 + |nf|)^p + \ln \frac{-\varphi - \dfrac{\delta}{2} |z|^2}{ \sqrt{- \varphi (\mathcal{Z}) - s_0 }}\exp \left(\frac{1}{p}\ln\frac{-\varphi - \dfrac{\delta}{2} |z|^2}{ \sqrt{- \varphi (\mathcal{Z}) - s_0 }} + C(p) \right)\right) \omega^n \\
%  \leq & \int_{\Omega_{s_0}} \left(e^{nf} (1 + |nf|)^p + C \ln \frac{-\varphi - \dfrac{\delta}{2} |z|^2}{ \sqrt{- \varphi (\mathcal{Z}) - s_0 }} \left(\frac{-\varphi - \dfrac{\delta}{2} |z|^2}{ \sqrt{- \varphi (\mathcal{Z}) - s_0 }}  \right)^{\frac{1}{p}}\right) \omega^n \\
  \leq & \int_{\Omega_{s_0}} \left(e^{nf} (1 + |nf|)^p + C + C \left(\frac{-\varphi - \dfrac{\delta}{2} |z|^2}{ \sqrt{- \varphi (\mathcal{Z}) - s_0 }} \right) \right) \omega^n \\
%  \leq & \int_{\Omega} \left(e^{nf} (1 + |nf|)^p + C + C \left(\frac{-\varphi - \dfrac{\delta}{2} |z|^2}{ \sqrt{- \varphi (\mathcal{Z}) - s_0 }} \right)\right) \omega^n \\
%  \leq & \Vert e^{nf}\Vert_{L^1(log L)^p} + C + C \int_\Omega \left(\frac{-\varphi - \dfrac{\delta}{2} |z|^2}{ \sqrt{- \varphi (\mathcal{Z}) - s_0 }} \right)   \omega^n \\
  \leq &\, \Vert e^{nf}\Vert_{L^1(log L)^p} + C + C  \frac{\Vert \varphi\Vert_{L^1}}{ \sqrt{- \varphi (\mathcal{Z}) - s_0 }}   .
  \end{aligned}
 \end{equation}
 The third line in \eqref{inequality-3-18} is from generalized Young's inequality
 \begin{equation*}
 \begin{aligned}
 	u^p e^{nf} 
 %	&\leq   e^{nf} (\ln(1 + e^{{|n f|}} ))^p + e^{u} u^p \\
 	&\leq  e^{nf} (1 + |nf|)^p + e^{u} u^p 
 %	\\
 %	&\leq \exp({\mathcal{F}_t(z)}) (1 + |\mathcal{F}_t(z)|^p) (1 + 1)^{p - 1}+ C(p) e^{2 v(z)} \\
 %	&\leq 2^{p - 1}e^{\mathcal{F}_t(z)}  |\mathcal{F}_t(z)|^p + 2^{p - 1}  e^{\mathcal{F}_t(z)} + C(p) e^{2 v(z)}  
 .
 \end{aligned} 
 \end{equation*}
 Since $ \chi + \delta \omega + \sqrt{-1} \partial\bar\partial \varphi \in \Gamma^1$, $\Vert\varphi\Vert_{L^1}$ is uniformly bounded~\cite{TosattiWeinkove2010}\cite{TosattiWeinkove2013}.
 Therefore, $-\varphi$ is bounded from above.

 \end{proof}

%\newpage

\medskip

\subsection{Stability estimate}

We plan to study the stability estimate of the following equation,
\begin{equation}
\label{equation-3-19}
	F(\chi + \delta\omega + \sqrt{-1}  \partial\bar\partial \varphi) =  e^f , \quad \sup_M \varphi = 0.
\end{equation}
%Obviously, the equation includes Hessian equations.
The stability estimate is a particular case of $L^\infty$ estimate. The crucial step is to read out appropriate structure condition~\eqref{structure-condition-1}. 
%
%
%First, we need a stability estimate.
\begin{theorem}
\label{theorem-3-3}
The function $F$  is concave, and satisfies that  there is a constant $c_1 > 0$ so that 
wherever $\hat\chi \geq 0$, we must have
\begin{equation}
\label{condition-3-20}
F (\chi + \hat\chi) \geq F (\hat\chi) \geq c_1 \left(\frac{\hat\chi^n}{\omega^n}\right)^{\frac{1}{n}} .
\end{equation}
Assume that we have
\begin{equation*}
F(\chi + \delta\omega + \sqrt{-1} \partial\bar\partial \varphi_1) \leq e^{f_1} ,  \quad\sup_M \varphi_1 = 0
\end{equation*}
and 
\begin{equation*}
F(\chi + \delta\omega + \sqrt{-1} \partial\bar\partial \varphi_2) \geq 0 ,  \qquad\sup_M \varphi_2 \leq 0.
\end{equation*}
Suppose that for $q > 1$, there is a constant $K > 0$ %and $L > 0$ 
so that
\begin{equation*}
\Vert e^{nf_1} \Vert_{L^q} \leq K
% \quad \text{ and } \quad \Vert e^{nf_2} \Vert_{L^q} \leq K 
%\quad \text{ and } \quad \Vert \varphi_1 \Vert_{L^q} \leq L
.
\end{equation*}
%Then for $p > n$, there exists a constant $C > 0$ such that 
%\begin{equation}
%	 \sup_M (\varphi_2 -\varphi_1 ) \leq C \left\Vert (-\varphi_1 + \varphi_2)^+ \right\Vert_{L^{q^*}}^{\frac{1}{n + 1 + \frac{np}{p - n} + \epsilon}} .
%\end{equation}
Then for any $\epsilon > 0$ and $q' > 0$, there exists a constant $C > 0$  such that
\begin{equation}
\label{inequality-3-21}
	 \sup_M (\varphi_2 -\varphi_1 ) \leq C \left\Vert (\varphi_2 -\varphi_1 )^+ \right\Vert_{L^{q'}}^{\frac{q'}{n q^* + q' + \epsilon}} ,
\end{equation}
where $q^* = \frac{q}{q - 1}$.

\end{theorem}

\begin{proof}

If $\Vert (\varphi_2 - \varphi_1)^+\Vert_{L^{q'}} = 0$, it must be that
\begin{equation*}
	\varphi_2 \leq \varphi_1 \quad \text{ on } M ,
\end{equation*}  
and thus Inequality~\eqref{inequality-3-21} is definitely true  for arbitrarily constant $C$. 
Without loss of generality, we may assume that 
\begin{equation*}
\Vert (\varphi_2 - \varphi_1)^+\Vert_{L^{q'}} > 0 .
\end{equation*}

At any point $z \in M$ with $\hat\chi (z) \geq 0$ and 
\begin{equation*}
	F (\chi (z) + \hat\chi (z)) \leq e^{f_1},
\end{equation*}
we have from \eqref{condition-3-20}
\begin{equation*}
	\hat \chi^n (z) \leq \frac{ e^{n f_1} }{c^n_1} \omega^n .
\end{equation*}
By Theorem~\ref{theorem-3-1}, $\varphi_1$ is uniformly bounded.

%\newpage

By concavity of $F$, for $ 0< r < \frac{1}{2}$, 
\begin{equation}
\label{inequality-3-22}
\begin{aligned}
	e^{f_1}
%	&= F \left( \chi + \delta \omega + i\partial\bar\partial \varphi_1\right) \\
%	&= F \left( \chi + \delta \omega + (1 - r) i\partial\bar\partial \varphi_2 - (1 - r) i\partial\bar\partial \varphi_2 + i \partial\bar\partial \varphi_1\right) \\
%	&= F \left((1 - r) \left(\chi + \delta\omega + i\partial\bar\partial\varphi_2\right) + r \left( \chi + \delta \omega + \frac{1}{r} i\partial\bar\partial\varphi_1 - \frac{1 - r}{r} i\partial\bar\partial\varphi_2\right)\right) \\
	&\geq	(1 - r) F (\chi + \delta \omega + \sqrt{-1} \partial\bar\partial\varphi_2) \\
	&\qquad + r F \left( \chi + \delta \omega +  \frac{1}{r} \sqrt{-1} \partial\bar\partial\varphi_1 - \frac{1 - r}{r} \sqrt{-1} \partial\bar\partial\varphi_2 \right) \\
%	&= 	(1 - r) e^{f_2} + r F \left( \chi + \delta \omega +  \frac{1}{r} i\partial\bar\partial\varphi_1 - \frac{1 - r}{r} i\partial\bar\partial\varphi_2 \right) \\
	&\geq  r F \left( \delta \omega +  \frac{1}{r} \sqrt{-1} \partial\bar\partial\varphi_1 - \frac{1 - r}{r} \sqrt{-1} \partial\bar\partial\varphi_2 \right) ,
\end{aligned}
\end{equation}
wherever $\delta \omega + \frac{1}{r} \sqrt{-1} \partial\bar\partial\varphi_1 - \frac{1 - r}{r} \sqrt{-1} \partial\bar\partial\varphi_2  \geq 0$.
Combining \eqref{inequality-3-22} with \eqref{condition-3-20},
%\begin{equation}
%	e^{f_1}
%	\geq  r c_1\left( \frac{ \left( \delta \omega +  \frac{1}{r} \sqrt{-1} \partial\bar\partial\varphi_1 - \frac{1 - r}{r} \sqrt{-1} \partial\bar\partial\varphi_2 \right)^n }{\omega^n} \right)^{\frac{1}{n}} 
%	,
%\end{equation}
%and hence
\begin{equation*}
	\frac{e^{ nf_1}}{c^n_1 r^n} \omega^n \geq  \left( \delta \omega +  \frac{1}{r} \sqrt{-1} \partial\bar\partial\varphi_1 - \frac{1 - r}{r} \sqrt{-1} \partial\bar\partial\varphi_2 \right)^n .
\end{equation*}

As in the proof of Theorem~\ref{theorem-3-1}, there is a constant $r_0 > 0$ so that for all $\mathcal{Z} \in M$, there is a local chart satisfying $\omega_{i\bar j} (\mathcal{Z}) = \delta_{ij}$ and
 \begin{equation*}
 \frac{1}{2} \sqrt{-1}  \sum_{i,j} \delta_{ij} dz^i \wedge d\bar z^j \leq \omega \leq 2 \sqrt{-1} \sum_{i,j} \delta_{ij} dz^i \wedge d\bar z^j \quad \text{ in } B(\mathcal{Z},2r_0).
 \end{equation*}
We choose  the maximum point of $\frac{1 - r}{r} \varphi_2 - \frac{1}{r} \varphi_1$ on $M$ as $\mathcal{Z}$,
where $0 < r < \frac{1}{2}$ is to be specified later. 
Let
\begin{equation*}
	u_s (z) = \frac{1}{r}\varphi_1 (z) - \frac{1 - r}{r} \varphi_2 (z) - \frac{1}{r} \varphi_1 (\mathcal{Z}) + \frac{1 - r}{r} \varphi_2 (\mathcal{Z}) + \frac{\delta}{2} |z|^2 - s ,
\end{equation*}
where $0 < s < s_0 := 2 \delta r^2_0$, and denote
\begin{equation*}
	\Omega := B (\mathcal{Z},2r_0) , \quad \Omega_s : = \{ u_s (z) < 0\} .
\end{equation*}

We solve the Dirichlet problem
 \begin{equation*}
 \left\{
 \begin{aligned}
 \left(\sqrt{-1} \partial \bar\partial \psi_{s,k}\right)^n &= \frac{\tau_k (- u_s)}{A_{s,k}} \frac{e^{nf_1}}{c^n_1 r^n} \omega^n , &&\text{ in } \Omega ,\\
 \psi_{s,k} &= 0, && \text{ on } \partial\Omega ,
 \end{aligned}
 \right.
 \end{equation*}
 where  \begin{equation*}
   A_{s,k} := \int_\Omega \tau_k (- u_s) \frac{e^{nf_1}}{c^n_1 r^n}  \omega^n \to A_s := \int_{\Omega_s} (- u_s) \frac{e^{nf_1}}{c^n_1 r^n}  \omega^n ,
 \end{equation*}
   and
  $\tau_k$ is defined in the proof of Theorem~\ref{theorem-3-1}.

Define 
\begin{equation*}
	\Phi := - \epsilon (- \psi_{s,k})^{\frac{n}{n + 1}} - u_s ,
\end{equation*}
where
\begin{equation*}
	\epsilon := \left(\frac{n}{n + 1}\right)^{-\frac{n}{n + 1}} A^{\frac{1}{n + 1}}_{s,k} .
\end{equation*}
The function $\Phi$ is continuous, and hence must archive $\sup_{\bar\Omega} \Phi$ at some point $z_{max} \in \bar\Omega$. If $z_{max} \in \bar\Omega\setminus \Omega_s$,
\begin{equation*}
	\Phi (z_{max}) \leq - \epsilon (-\psi_{s,k})^{\frac{n}{n + 1}} < 0 .
\end{equation*}
If $z_{max} \in \Omega_s$,
\begin{equation*}
\begin{aligned}
	0
	&\geq
	[\partial_i\bar\partial_j \Phi] \\
	&\geq
	\left[\frac{\epsilon n}{n + 1} (- \psi_{s,k})^{-\frac{1}{n + 1}} \partial_i\bar\partial_j \psi_{s,k} - \partial_i \bar\partial_j u_s \right] 
	\\
	&\geq 
	\left[\frac{\epsilon n}{n + 1} (- \psi_{s,k})^{-\frac{1}{n + 1}} \partial_i\bar\partial_j \psi_{s,k} - \partial_i \bar\partial_j \left(\frac{1}{r} \varphi_1 - \frac{1 - r}{r} \varphi_2 \right) - \delta \omega_{i\bar j}\right] 
	,
\end{aligned}
\end{equation*}
and thus 
\begin{equation}
\label{inequality-3-29}
	\delta \omega +   \frac{1}{r} \sqrt{-1} \partial \bar\partial  \varphi_1 - \frac{1 - r}{r} \sqrt{-1}  \varphi_2 
	\geq \frac{\epsilon n}{n + 1} (- \psi_{s,k})^{- \frac{1}{n + 1}} \sqrt{-1} \partial\bar\partial \psi_{s,k} .
\end{equation}
Considering the volume form of two sides of Inequality~\eqref{inequality-3-29},
\begin{equation*}
	1 
	\geq \epsilon^n \left(\frac{n}{n + 1}\right)^n (- \psi_{s,k})^{-\frac{n}{n + 1}}\frac{\tau_k (- u_s)}{A_{s,k}} 
	\geq 
%	\left( \left(\frac{n}{n + 1}\right)^{-\frac{n}{n + 1}} A^{\frac{1}{n + 1}}_{s,k} \right)^n \left(\frac{n}{n + 1}\right)^n (- \psi_{s,k})^{-\frac{n}{n + 1}}\frac{ - u_s}{A_{s,k}} 
%	= 
%	\left(\frac{n}{n + 1}\right)^{\frac{n}{n + 1}} A^{- \frac{1}{n + 1}}_{s,k} (- \psi_{s,k})^{-\frac{n}{n + 1}} ( - u_s ) 
%	= 
	\epsilon^{-1}  (- \psi_{s,k})^{-\frac{n}{n + 1}} ( - u_s ) 
	,
\end{equation*}
that is,
\begin{equation*}
	- u_s - \epsilon (- \psi_{s,k})^{\frac{n }{n + 1}} \leq 0.
\end{equation*}
In sum, we have that on $\bar \Omega$
\begin{equation*}
	- u_s \leq \epsilon (- \psi_{s,k})^{\frac{n}{n + 1}} = \left(\frac{n + 1}{n}\right)^{\frac{n}{n + 1}} A^{\frac{1}{n + 1}}_{s,k} (- \psi_{s,k})^{\frac{n}{n + 1}} .
\end{equation*}
There are constants $\beta > 0$ and $C > 0$ such that
\begin{equation*}
	\int_{\Omega_s} \exp \left( \beta\frac{(- u_s)^{\frac{n + 1}{n}} }{A^{\frac{1}{n}}_{s,k}}\right) \omega^n 
	\leq \int_\Omega \exp \left( - \beta \frac{n + 1}{n}  \psi_{s,k} \right) \omega^n 
	\leq C.
\end{equation*} 
Letting $k \to \infty$,
\begin{equation*}
	\int_{\Omega_s}  \exp \left( \beta\frac{(- u_s)^{\frac{n + 1}{n}} }{A^{\frac{1}{n}}_{s}}\right) \omega^n 
		  \leq C ,
\end{equation*}
and consequently for any $p > 1$, 
\begin{equation}
\label{inequality-3-35}
	\int_{\Omega_s} (- u_s)^p \omega^n   \leq C A^{\frac{ p}{n + 1}}_s .
\end{equation}
An appropriate coefficient $p$ is to be determined later. 
Applying H\"older inequality to $A_s$,
\begin{equation}
\label{inequality-3-36}
\begin{aligned}
	A_s 
	&= \frac{1}{c^n_1 r^n} \int_{\Omega_s} (- u_s) e^{nf_1} \omega^n \\
	&\leq
	\frac{1}{c^n_1 r^n} \left(\int_{\Omega_s} e^{q nf_1} \omega^n\right)^{\frac{1}{p (q - 1)}} \left(\int_{\Omega_s} (- u_s)^{p} \omega^n\right)^{\frac{1}{p}} \left(\int_{\Omega_s} e^{n f_1} \omega^n\right)^{1 - \frac{1}{p} - \frac{1}{p (q - 1)}}\\
	&\leq 
	\frac{C}{ r^n} \Vert e^{n f_1}\Vert^{\frac{q}{p (q - 1)}}_{L^q} 
	A^{\frac{1}{n + 1}}_s 
	\left(\int_{\Omega_s} e^{n f_1} \omega^n\right)^{1 - \frac{q}{p(q - 1)}}
	.
\end{aligned}
\end{equation}
The this line in \eqref{inequality-3-36} is from \eqref{inequality-3-35}. Rewriting \eqref{inequality-3-36}, we obtain
\begin{equation*}
	A_s 
	\leq \frac{C}{r^{n + 1}} \Vert e^{nf_1}\Vert^{\frac{q (n + 1)}{p (q - 1) n}}_{L^q}	\left(\int_{\Omega_s} e^{n f_1} \omega^n\right)^{\left(1 - \frac{q}{p (q - 1)} \right) \frac{n + 1}{n}} 
	.
\end{equation*}
Therefore, for $0 < t < s < s_0$,
\begin{equation*}
\begin{aligned}
	t \int_{\Omega_{s - t}} e^{n f_1} \omega^n 
	&\leq \int_{\Omega_{s - t}} (- u_s) e^{n f_1} \omega^n \\
	&\leq 
	\int_{\Omega_s} (- u_s) e^{n f_1} \omega^n \\
%	&= 
%	c^n_1 r^n A_s \\
	&\leq
	\frac{C}{r } \Vert e^{nf_1}\Vert^{\frac{q (n + 1)}{p (q - 1) n}}_{L^q}	\left(\int_{\Omega_s} e^{n f_1} \omega^n\right)^{\left(1 - \frac{q}{p (q - 1)} \right) \frac{n + 1}{n}} 
		.
\end{aligned}
\end{equation*}
For any $0 < \delta_0 < \frac{1}{n}$, we can choose $p > 1$ sufficiently large so that
\begin{equation*}
	t \int_{\Omega_{s - t}} e^{n f_1} \omega^n 
	\leq
	\frac{C}{r } \Vert e^{nf_1}\Vert^{\frac{1}{n} - \delta_0}_{L^q}	\left(\int_{\Omega_s} e^{n f_1} \omega^n\right)^{1 + \delta_0} 
		.
\end{equation*}
By Lemma~\ref{lemma-3-2}, there is a constant $c_0 > 0$ such that
\begin{equation}
\label{inequality-3-21-1}
	\int_{\Omega_{s_0}} e^{n f_1} \omega^n \geq c_0 \frac{r^{\frac{  1}{\delta_0}}}{\Vert e^{n f_1}\Vert^{\frac{1}{n \delta_0} - 1}_{L^q}} .
\end{equation}

When
\begin{equation*}
	\Vert (\varphi_2 - \varphi_1)^+\Vert_{L^{q'}} \geq \frac{1}{2^{\frac{q^*}{\delta_0 q'} + 1}} ,
\end{equation*}
we have
\begin{equation*}
	\sup_M(\varphi_2 - \varphi_1)
	\leq \Vert\varphi_1\Vert_{L^\infty} 
	\leq 2 \Vert \varphi_1 \Vert_{L^\infty} 	\Vert (\varphi_2 - \varphi_1)^+\Vert^{\frac{\delta_0 q'}{q^* + \delta_0 q'} }_{L^{q'}} .
\end{equation*}

When
\begin{equation*}
	\Vert (\varphi_2 - \varphi_1)^+\Vert_{L^{q'}} < \frac{1}{2^{\frac{q^*}{\delta_0 q'} + 1}} ,
\end{equation*}
we define 
\begin{equation*}
	r: = 	\Vert (\varphi_2 - \varphi_1)^+\Vert^{\frac{\delta_0 q'}{q^* + \delta_0 q'} }_{L^{q'}} < \frac{1}{2} .
\end{equation*}
If 
$
 \varphi_1 (\mathcal{Z}) - (1 - r) \varphi_2 (\mathcal{Z})  \geq - (s_0 + 1 + \Vert\varphi_1\Vert_{L^\infty} ) r
$,
%\begin{equation*}
%\frac{1}{r}\varphi_1 (\mathcal{Z}) - \frac{1 - r}{r} \varphi_2 (\mathcal{Z})  \geq - s_0 - 1 -\Vert\varphi_1\Vert_{L^\infty} ,
%\end{equation*}
then
%\begin{equation*}
%\frac{1}{r}\varphi_1 (z) - \frac{1 - r}{r} \varphi_2 (z) \geq \frac{1}{r}\varphi_1 (\mathcal{Z}) - \frac{1 - r}{r} \varphi_2 (\mathcal{Z})  \geq - s_0 - 1 -\Vert\varphi_1\Vert_{L^\infty} , \forall z \in M
%\end{equation*}
%and consequently
%%\begin{equation}
%%	(1 - r) (\varphi_2 (z) - \varphi_1 (z)) \leq r (s_0 + 1 + \Vert\varphi_1 \Vert_{L^\infty}) .
%%\end{equation}
\begin{equation*}
	\sup_M (\varphi_2  - \varphi_1 ) \leq  2 (s_0 + 1 + \Vert\varphi_1 \Vert_{L^\infty})  r = 2 (s_0 + 1 + \Vert\varphi_1 \Vert_{L^\infty})	\Vert (\varphi_2 - \varphi_1)^+\Vert^{\frac{\delta_0 q'}{q^* + \delta_0 q'} }_{L^{q'}} .
\end{equation*}
Otherwise,
%\begin{equation*}
%\frac{1}{r}\varphi_1 (\mathcal{Z}) - \frac{1 - r}{r} \varphi_2 (\mathcal{Z})  < - s_0 - 1 -\Vert\varphi_1\Vert_{L^\infty} .
%\end{equation*}
on $\Omega_{s_0}$
\begin{equation}
\label{inequality-3-22-1}
 \frac{1}{r}\varphi_1 (z) - \frac{1 - r}{r} \varphi_2 (z) + \frac{\delta}{2} |z|^2 <  \frac{1}{r} \varphi_1 (\mathcal{Z}) - \frac{1 - r}{r} \varphi_2 (\mathcal{Z})  + s_0 < - 1 -\Vert\varphi_1\Vert_{L^\infty}  .
\end{equation}
By  the second inequality in \eqref{inequality-3-22-1}  and   \eqref{inequality-3-21-1}, we derive that
\begin{equation}
\label{inequality-3-23-1}
\begin{aligned}
&\quad	
c_0  r^{\frac{1}{\delta_0}}  \left( \left((1 - r) \varphi_2 -   \varphi_1 \right)^+ - \left(s_0 + \Vert\varphi_1\Vert_{L^\infty} \right) r\right)^{\frac{q'}{q^*}} \\
	&\leq
	c_0  r^{\frac{1}{\delta_0}}  \left((1 - r) \varphi_2 (\mathcal{Z}) -   \varphi_1 (\mathcal{Z})  - \left(s_0 + \Vert\varphi_1\Vert_{L^\infty} \right) r\right)^{\frac{q'}{q^*}} \\
	&\leq 
	\Vert e^{n f_1}\Vert^{\frac{1}{n \delta_0} - 1}_{L^q} \left(\int_{\Omega_{s_0} } e^{n f_1}\omega^n\right) \left((1 - r) \varphi_2 (\mathcal{Z}) -   \varphi_1 (\mathcal{Z})  - \left( s_0 + \Vert\varphi_1\Vert_{L^\infty} \right) r \right)^{\frac{q'}{q^*}} 
	.
\end{aligned}
\end{equation}
Substituting the first inequality in \eqref{inequality-3-22-1} into \eqref{inequality-3-23-1}, 
\begin{equation}
\label{inequality-3-24-1}
\begin{aligned}
&\quad	
c_0  r^{\frac{1}{\delta_0}}  \left( \left((1 - r) \varphi_2 -   \varphi_1 \right)^+ - \left(s_0 + \Vert\varphi_1\Vert_{L^\infty} \right) r\right)^{\frac{q'}{q^*}} \\
%	&\leq
%	c_0  r^{\frac{1}{\delta_0}}  \left((1 - r) \varphi_2 (\mathcal{Z}) -   \varphi_1 (\mathcal{Z})  - \left(s_0 + \Vert\varphi_1\Vert_{L^\infty} \right) r\right)^{\frac{q'}{q^*}} \\
%	&\leq 
%	\Vert e^{n f_1}\Vert^{\frac{1}{n \delta_0} - 1}_{L^q} \left(\int_{\Omega_{s_0} } e^{n f_1}\omega^n\right) \left((1 - r) \varphi_2 (\mathcal{Z}) -   \varphi_1 (\mathcal{Z})  - \left( s_0 + \Vert\varphi_1\Vert_{L^\infty} \right) r \right)^{\frac{q'}{q^*}} \\
	&< 
	\Vert e^{n f_1}\Vert^{\frac{1}{n \delta_0} - 1}_{L^q}
	\int_{\Omega_{s_0}}   \left( (1 - r) \varphi_2 -   \varphi_1 - \left(\frac{\delta}{2} |z|^2 + \Vert\varphi_1\Vert_{L^\infty} \right) r  \right)^{\frac{q'}{q^*}} e^{nf_1} \omega^n  .
\end{aligned}
\end{equation}
Applying H\"older inequality to \eqref{inequality-3-24-1},
\begin{equation}
\label{inequality-3-45}
\begin{aligned}
&\quad	
c_0  r^{\frac{1}{\delta_0}}  \left( \left((1 - r) \varphi_2 -   \varphi_1 \right)^+ - \left(s_0 + \Vert\varphi_1\Vert_{L^\infty} \right) r\right)^{\frac{q'}{q^*}} \\
%	&\leq
%	c_0  r^{\frac{1}{\delta_0}}  \left((1 - r) \varphi_2 (\mathcal{Z}) -   \varphi_1 (\mathcal{Z})  - \left(s_0 + \Vert\varphi_1\Vert_{L^\infty} \right) r\right)^{\frac{q'}{q^*}} \\
%	&\leq 
%	\Vert e^{n f_1}\Vert^{\frac{1}{n \delta_0} - 1}_{L^q} \left(\int_{\Omega_{s_0} } e^{n f_1}\omega^n\right) \left((1 - r) \varphi_2 (\mathcal{Z}) -   \varphi_1 (\mathcal{Z})  - \left( s_0 + \Vert\varphi_1\Vert_{L^\infty} \right) r \right)^{\frac{q'}{q^*}} \\
	&< 
%	\Vert e^{n f_1}\Vert^{\frac{1}{n \delta_0} - 1}_{L^q}
%	\int_{\Omega_{s_0}}   \left( (1 - r) \varphi_2 -   \varphi_1 - \left(\frac{\delta}{2} |z|^2 + \Vert\varphi_1\Vert_{L^\infty} \right) r  \right)^{\frac{q'}{q^*}} e^{nf_1} \omega^n  \\
%	&\leq 
	\Vert e^{n f_1}\Vert^{\frac{1}{n \delta_0}  }_{L^q}
	 \left( \int_{\Omega_{s_0}} \left( (1 - r) \varphi_2 -  \varphi_1 - \left(\frac{\delta}{2}  |z|^2 + \Vert\varphi_1\Vert_{L^\infty} \right) r \right)^{q'} \omega^n\right)^{\frac{1}{q^*}} \\
%	&\leq 
%	\Vert e^{n f_1}\Vert^{\frac{1}{n \delta_0} }_{L^q}
%	  \left( \int_M \left(\left( (1 - r) (\varphi_2 -  \varphi_1) - r \varphi_1 - \left(\frac{\delta }{2} |z|^2 +\Vert\varphi_1\Vert_{L^\infty}\right) r \right)^+\right)^{q'} \omega^n\right)^{\frac{1}{q^*}} \\
	&\leq 
	\Vert e^{n f_1}\Vert^{\frac{1}{n \delta_0} }_{L^q}
	  \left( \int_M \left( \left( (1 - r) (\varphi_2 -  \varphi_1)   \right)^+\right)^{q'} \omega^n\right)^{\frac{1}{q^*}} \\
%	&\leq \Vert e^{n f_1}\Vert^{\frac{1}{n \delta_0} }_{L^q}
%	  \left(\int_M (1 - r)^{q'} \left((\varphi_2 -  \varphi_1)^+ \right)^{q'}    \omega^n\right)^{\frac{1}{q^*}} \\
%	&\leq \Vert e^{n f_1}\Vert^{\frac{1}{n \delta_0} }_{L^q}
%	  \left( \int_M \left( (\varphi_2 -  \varphi_1)^+ \right)^{q'}    \omega^n\right)^{\frac{1}{q^*}}   \\
	&\leq \Vert e^{n f_1}\Vert^{\frac{1}{n \delta_0} }_{L^q}
	   \Vert (\varphi_2 -  \varphi_1)^+     \Vert^{\frac{q'}{q^*}}_{L^{q'}} 
	   ,
\end{aligned}
\end{equation}
and hence we obtain by rewriting \eqref{inequality-3-45}% and the fact that $r < \frac{1}{2}$
,
\begin{equation*}
\begin{aligned}
	 \sup_M (\varphi_2  - \varphi_1 )
	 &\leq C \Vert e^{n f_1}\Vert^{\frac{q^*}{n \delta_0 q'} }_{L^q}    \frac{ \Vert (\varphi_2 -  \varphi_1)^+     \Vert_{L^{q'}}}{r^{\frac{q^*}{\delta_0 q'}}} + 2 (s_0 + \Vert \varphi_1 \Vert_{L^\infty}) r \\
	 &= \left(C \Vert e^{n f_1}\Vert^{\frac{q^*}{n \delta_0 q'} }_{L^q}   + 2 (s_0 + \Vert \varphi_1 \Vert_{L^\infty}) \right) 	\Vert (\varphi_2 - \varphi_1)^+\Vert^{\frac{\delta_0 q'}{q^* + \delta_0 q'} }_{L^{q'}} 
	 .
\end{aligned}
\end{equation*}

\end{proof}

\medskip

\subsection{Viscosity solution}

As the positivity of $\chi + \sqrt{-1} \partial\bar\partial \varphi$ is dependent on $\omega$ if the equation is not of complex Monge-Amp\`ere type, we cannot localize the equation to a domain in $\mathbb{C}^n$ and apply Hartogs Lemma. 
%We do not have the $L^1$ compactness in hand. 
Next, we establish a compact embedding result.
\begin{lemma}
\label{lemma-3-4}
$L^\infty$ bounded  solution set is precompact in  $L^{q'}$ norm for $1 \leq q' < + \infty$.
 
\end{lemma}
\begin{proof}

 Since $(\chi + \sqrt{-1} \partial\bar\partial \varphi) \wedge \omega^{n - 1} \geq 0$ and $ \varphi \leq 0$, we have
 \begin{equation*}
 \label{equation-3-2}
 \begin{aligned}
 	\Vert \varphi \Vert_{L^\infty} \int_M \chi \wedge \omega^{n - 1}
 	&\geq
 	\int_M - \varphi (\chi + \sqrt{-1} \partial\bar\partial \varphi) \wedge \omega^{n - 1}  \\
 	&\geq \int_M -\varphi \sqrt{-1}  \partial\bar\partial \varphi \wedge \omega^{n - 1} \\
 	&= \Vert \nabla \varphi\Vert^2_{L^2} .
 \end{aligned}
 \end{equation*}
 If $\{\varphi_\alpha\}$ is also a bounded set in $L^\infty$ norm, then %by compact embedding 
 we know that $\{\varphi_\alpha\}$ is precompact in $L^1$ norm. 
From a $L^1$ convergent sequence $\{\varphi_i\}$, we can further pick a subsequence $\{\varphi_{i_j}\}$ which converges almost everywhere. 
 Therefore for any fixed $1 \leq q' < +\infty$,
 \begin{equation*}
 \begin{aligned}
 	\int_M |\varphi_{i_j} - \varphi_{i_k}|^{q'} \omega^n 
 	&\leq 	\Vert\varphi_{i_j} - \varphi_{i_k}\Vert^{q' -1}_{L^\infty}\int_M |\varphi_{i_j} - \varphi_{i_k}| \omega^n \\
 	&\leq C \int_M |\varphi_{i_j} - \varphi_{i_k}| \omega^n \\
 	&\to 0 \quad  (j,k \to  + \infty) , 
 \end{aligned}
\end{equation*}
which implies $\{\varphi_{i_j}\}$ is Cauchy in $L^{q'}$.

\end{proof}

%\newpage

With Lemma~\ref{lemma-3-4}, there is a standard procedure to construct viscosity solutions~\cite{Lions1983}.
\begin{theorem}
\label{theorem-3-5}
%The function $\mathcal{F}$  is concave, and   $\forall$ $1 \leq i \leq n$
%\begin{equation*}
%\frac{\partial \mathcal{F}}{\partial \lambda_i } \geq 0 .
%\end{equation*}
The function $F$ is concave, and satisfies that
there is a constant $c_1 > 0$ so that 
wherever $\hat\chi \geq 0$, we  have
\begin{equation*}
F (\chi + \hat\chi) \geq F (\hat\chi) \geq c_1 \left(\frac{\hat\chi^n}{\omega^n}\right)^{\frac{1}{n}} .
\end{equation*}
Suppose that  $e^f \in C (M)$ satisfies that  there is a sequence $\{e^{f_i}\} \subset C (M) $   uniformly  convergent to $e^f$ 
so that 
for each $i$ there is a $C^2 $ classical solution $\varphi_i$ solving
\begin{equation*}
F (\chi + \delta \omega + \sqrt{-1} \partial\bar\partial \varphi_i) = e^{f_i} , \quad \sup_M \varphi_i = 0.
\end{equation*}
Then there is a viscosity solution to Equation~\eqref{equation-3-19}.

\end{theorem}

\begin{proof}

Since $e^{f_j} \to e^f$ in $L^p$ for $p > n$, $\{\varphi_j\}$ are uniformly bounded. Lemma~\ref{lemma-3-4} tells us that $\{\varphi_j\}$ is convergent in $L^1$ norm by passing to a subsequence.

By Theorem~\ref{theorem-3-3}, there is a constant $C > 0$ such that
\begin{equation*}
	\sup_M (\varphi_j - \varphi_i) 
	\leq C \Vert (\varphi_j - \varphi_i)^+\Vert^{\frac{1}{n q^* + 2}}_{L^1} 
	\leq C \Vert \varphi_j - \varphi_i\Vert^{\frac{1}{n q^* + 2}}_{L^1}
	.
\end{equation*}
Switching $i$ and $j$, we hvae
\begin{equation}
\label{inequality-3-53}
	\Vert \varphi_j - \varphi_i\Vert_{L^\infty} \leq C \Vert \varphi_j - \varphi_i\Vert^{\frac{1}{n q^* + 2}}_{L^1} .
\end{equation}
We can choose inductively an increasing sequence $\{i_k \}$ such that $\forall j  \geq i_k$,
\begin{equation*}
	\Vert \varphi_j - \varphi_{i_k}\Vert_{L^1} \leq \frac{1}{2^{(n q^* + 2) k}} ,
\end{equation*}
and hence by \eqref{inequality-3-53} ,
\begin{equation}
	\Vert \varphi_j - \varphi_{i_k}\Vert_{L^\infty} \leq \frac{1}{2^k} .
\end{equation}
Therefore $\{\varphi_{i_k}\}$ is a Cauchy sequence in $L^\infty$ norm, which is uniformly convergent to a function $\varphi \in C(M)$.

For any point $z_0 \in M$ and any $C^2$ function $u$ such that $u(z_0) = \varphi(z_0)$ and
$u > \varphi$ in a punctured neighborhood of $z_0$. 
Choosing $\eta > 0$ small, we have $\max_{\partial B_{\eta} (z_0)} (\varphi - u) < 0$. For $j$ sufficiently large,
\begin{equation*}
\max_{\overline{B_\eta (z_0)}} (\varphi_j - u) 
%\geq \varphi_j (z_0) - \varphi (z_0) 
> \max_{\partial B_{\eta} (z_0) }(\varphi_j - u) .
\end{equation*}
So there exists a point $z_\delta \in B_\eta (z_0)$ with 
\begin{equation*}
\max_{\overline{B_\eta (z_0)}} (\varphi_j - u) = (\varphi_j - u)  (z_\eta).
\end{equation*}
We may pick a sequence $\eta_k \to 0 +$, and obtain $z_k \in B_{\eta_k} (z_0)$ and $\varphi_k $ such that
\begin{equation*}
\max_{\overline{B_{\eta_k} (z_0)}} (\varphi_k - u) = (\varphi_k- u)  (z_{k}) .
\end{equation*}
Therefore,
\begin{equation*}
F (\chi + \delta\omega + \sqrt{-1} \partial\bar\partial u) (z_k) \geq F (\chi + \delta\omega + \sqrt{-1}  \partial\bar\partial \varphi_k) (z_k) = e^{f_k} (z_k) .
\end{equation*}
Letting $k \to +\infty$, $F (\chi + \delta\omega + \sqrt{-1} \partial\bar\partial u) (z_0) \geq e^{f} (z_0)$. So $\varphi$ is a subsolution.

Similarly, we can also prove that $\varphi$ is a supersolution. 
%Since $\mathcal{F}$ is not defined on the complement of $\Gamma$, we need to be more careful.
For any point $z_0 \in M$ and any $C^2$ function $u$ such that $u(z_0) = \varphi(z_0)$ and
$u < \varphi$ in a punctured neighborhood of $z_0$. 
Choosing $\eta > 0$ small, we have $\min_{\partial B_{\eta} (z_0)} (\varphi - u) > 0$. For $j$ sufficiently large,
\begin{equation*}
\min_{\overline{B_\eta (z_0)}} (\varphi_j - u) 
%\leq \varphi_j (z_0) - \varphi (z_0) 
< \min_{\partial B_{\eta} (z_0) }(\varphi_j - u) .
\end{equation*}
So there exists a point $z_\delta \in B_\eta (z_0)$ with 
\begin{equation*}
\min_{\overline{B_\eta (z_0)}} (\varphi_j - u) = (\varphi_j - u)  (z_\eta).
\end{equation*}
We may pick a sequence $\eta_k \to 0 +$, and obtain $z_k \in B_{\eta_k} (z_0)$ and $\varphi_k $ such that
\begin{equation*}
\min_{\overline{B_{\eta_k} (z_0)}} (\varphi_k - u) = (\varphi_k- u)  (z_{k}) .
\end{equation*}
Since the eigenvalues of a matrix is Lipschitz continuous with respect to the entries after appropriate relabeling,
\begin{equation}
	\lim_{k \to \infty} \bm{\lambda} (\chi + \delta \omega + \sqrt{-1} \partial\bar\partial u) (z_k) = \bm{\lambda} (\chi + \delta \omega + \sqrt{-1} \partial\bar\partial u) (z_0) .
\end{equation}
If $\bm{\lambda} (\chi + \delta \omega + \sqrt{-1} \partial\bar\partial u) (z_0) \not\in \Gamma$, $\varphi$ is a supersolution. Otherwise, by observing that $\Gamma$ is open and truncating the sequence,
we obtain
\begin{equation*}
F (\chi + \delta\omega + \sqrt{-1} \partial\bar\partial u) (z_0) \leq e^{f} (z_0),
\end{equation*}
which also implies that $\varphi$ is a supersolution.

\end{proof}
%\begin{remark}
%Observing the simple fact that $\varphi_{i_k} + \frac{1}{2^{k - 1}}$ is decreasing and convergent to $\varphi$, it is possible to claim that $\varphi$ is a weak solution in pluripotential sense if the relevant pluripotential solution theory is well defined for equation form $F$.
%\end{remark}

\medskip

\subsection{Viscosity solution to complex Hessian equations}

%\begin{remark}
Theorem~\ref{theorem-3-5} tells us that the key to construct a viscosity solution is to find out an available smooth approximation. 
For complex Hessian equations, % and Hessian-quotient equations,
such an approximation is possible by the works of \cite{HouMaWu2010}\cite{DinewKolodziej2017}\cite{Sun2017u}\cite{Szekelyhidi}\cite{Zhang2017} even on Hermitian manifolds.
%\end{remark}
\begin{theorem}
\label{theorem-3-7}
Let $\chi \in \bar\Gamma^m_\omega$ be a closed real $(1,1)$-form on a compact K\"ahler manifold $(M,\omega)$ without boundary.  
Suppose that $e^f \in C(M)$ satisfies
\begin{equation*}
\int_M (\chi + \delta \omega)^m\wedge \omega^{n - m} = \int_M e^{mf} \omega^n
.
\end{equation*}
Then the complex Hessian equation
\begin{equation}
\label{equation-3-56}
(\chi + \delta \omega + \sqrt{-1} \partial\bar\partial \varphi)^m \wedge \omega^{n - m} = e^{mf} \omega^n, \quad \sup_M \varphi =0
\end{equation}
 admits a  viscosity solution.
\end{theorem}

\begin{proof}

By convolution and partition of unity, we can find a positive smooth approximation $e^{ m f_i}$ of $e^{m f} + \frac{1}{2^i}$ with $\Vert e^{m f_i} - e^{m f} - \frac{1}{2^i}\Vert_{L^\infty} < \frac{1}{2^{i}}$. We can find a smooth solution to
\begin{equation*}
	(\chi + \delta\omega + \sqrt{-1} \partial\bar\partial \varphi_i)^m \wedge \omega^{n - m} = e^{m b_i} e^{m f_i} \omega^n ,
\end{equation*}
where
\begin{equation*}
	\int_M (\chi + \delta\omega)^m \wedge \omega^{n - m} = e^{m b_i}\int_M e^{m f_i} \omega^n .
\end{equation*}
By comparing the integrals over $M$, we have
%\begin{equation}
% e^{m f} <	e^{m f_i} < e^{m f} + \frac{1}{2^{i - 1}}
%\end{equation}
%\begin{equation}
% \int_M e^{m f} \omega^n \int_M \omega^n < \int_M	e^{m f_i} \omega^n < \int_M e^{m f} \omega^n + \frac{1}{2^{i - 1}} \int_M \omega^n
%\end{equation}
%\begin{equation}
% \int_M e^{m f} \omega^n  < e^{m b_i} \int_M	e^{m f} \omega^n < \int_M e^{m f} \omega^n + \frac{1}{2^{i - 1}} \int_M \omega^n
%\end{equation}
\begin{equation*}
 1  < e^{m b_i}   < 1 + \frac{1}{2^{i - 1}} \frac{\int_M \omega^n}{\int_M (\chi + \delta \omega)^m \wedge \omega^{n - m}} ,
\end{equation*}
and hence
\begin{equation*}
\begin{aligned}
	&\quad \left\Vert e^{mb_i} e^{mf_i} - e^{mf} \right\Vert_{L^\infty} \\
	&\leq 
	e^{m b_i} \left\Vert e^{mf_i} - e^{mf}\right\Vert_{L^\infty} + \left(e^{mb_i} - 1\right) \left\Vert e^{mf}\right\Vert_{L^\infty} \\
%	&\leq
%	\frac{e^{m b_i}}{2^{i}}  + e^{m b_i}\left\Vert e^{mf_i} - e^{mf} - \frac{1}{2^i}\right\Vert_{L^\infty}  + \left(e^{mb_i} - 1\right) \left\Vert e^{mf}\right\Vert_{L^\infty} \\
	&\leq \frac{1}{2^{i - 1}} \left(1 + \frac{1}{2^{i - 1}} \frac{\int_M \omega^n}{\int_M (\chi + \delta \omega)^m \wedge \omega^{n - m}}\right) +  \frac{\left\Vert e^{mf}\right\Vert_{L^\infty}}{2^{i - 1}} \frac{\int_M \omega^n}{\int_M (\chi + \delta \omega)^m \wedge \omega^{n - m}} ,
\end{aligned}
\end{equation*}
which implies that $\{e^{b_i} e^{f_i}\}$ is also uniformly convergent to $e^f$. So we are able to construct a viscosity solution $\varphi$ to complex Hessian equation~\eqref{equation-3-56} according to Theorem~\ref{theorem-3-5}.

\end{proof}

\begin{remark}
A viscosity solution from Theorem~\ref{theorem-3-7} is also a weak solution in pluripotential sense, by realizing the facts that we can construct a decreasing function sequence from a uniformly convergent function sequence and that the pluripotential solution theory is well defined for complex Hessian equations when $\chi$ and $\omega$ are both closed~\cite{Blocki2005}\cite{DinewKolodziej2014}. 
\end{remark}

On Hermitian manifolds, we have to weaken the result. 
\begin{theorem}
Let $\chi \in \bar\Gamma^m_\omega$ be a real $(1,1)$-form on a compact Hermitian manifold $(M,\omega)$ without boundary.  
If $f \in C(M)$, 
then the complex Hessian equation
\begin{equation}
\label{equation-3-59}
(\chi + \delta \omega + \sqrt{-1} \partial\bar\partial \varphi)^m \wedge \omega^{n - m} = e^{m b} e^{mf} \omega^n, \quad \sup_M \varphi =0
\end{equation}
 admits a  viscosity solution pair $(\varphi, b)$.
\end{theorem}

\begin{proof}

We denote by
\begin{equation*}
	(\chi + \delta \omega)^m \wedge \omega^{n - m} = e^{m g} \omega^n .
\end{equation*}
%By comparison, 
%\begin{equation}
%	\min_M g \leq b + f \leq \max_M g ,
%\end{equation}
%and thus
%\begin{equation}
%	\min_M (g - f) \leq b \leq \max_M (g - f) .
%\end{equation}
There is a smooth approximation $f_i$ of $f$ with $\Vert f_i - f\Vert_{L^\infty} < \frac{1}{2^i}$. We can find a smooth solution pair $(\varphi_i,b_i)$ to
\begin{equation}
	(\chi + \delta \omega + \sqrt{-1} \partial\bar\partial \varphi_i)^m \wedge \omega^{n - m} = e^{m b_i} e^{m f_i} \omega^n
\end{equation}
By comparison,
\begin{equation}	
	\min_M (g - f) - \frac{1}{2^i} < \min_M (g - f_i) \leq b_i \leq \max_M (g - f_i) < \max_M (g - f) + \frac{1}{2^i} .
\end{equation}
By passing to a subsequence, $b_i$ is convergent to a constant $b$. As a result,
\begin{equation}
%\begin{aligned}
	\Vert b_i + f_i - b - f\Vert_{L^\infty}
	\leq |b_i - b| + \Vert   f_i   - f\Vert_{L^\infty} 
	\to 0 
	\qquad \text{ as } i \to \infty ,
%\end{aligned}
\end{equation}
which implies that $\left\{e^{b_i} e^{f_i}\right\}$ is uniformly convergent to $e^b e^f$. So there is a viscosity solution pair $(\varphi,b)$ to complex Hessian equation~\eqref{equation-3-59} by Theorem~\ref{theorem-3-5}.

\end{proof}

\medskip

\subsection{Pluripotential solution to complex Hessian equations on K\"ahler manifolds}

 By the $L^\infty$ estimate, we can improve the exponent in strong stability estimate of Hessian equations in \cite{SuiSun2021}. 
 To deal with the big case,  the exponent has to satisfy $q > \frac{n}{m}$ in \cite{SuiSun2021}. 
 We reformulate the strong stability estimate  as follows.
 \begin{theorem}
 \label{theorem-3-10}
 Let $\chi \in \bar\Gamma^m_\omega$ be a closed real $(1,1)$-form on a compact K\"ahler manifold $(M,\omega)$ without boundary.
 For Hessian equations
\begin{equation*}
 (\chi  + \delta\omega + \sqrt{-1}  \partial\bar\partial \varphi_1)^m \wedge \omega^{n - m} =    e^{m f_1} \omega^{n}, 
\end{equation*}
 and
\begin{equation*}
 (\chi  + \delta\omega + \sqrt{-1}  \partial\bar\partial \varphi_2)^m \wedge \omega^{n - m} =   e^{m f_2} \omega^{n } 
\end{equation*}
 where
\begin{equation*}
 \int_M e^{m f_1} \omega^n =\int_M e^{m f_2} \omega^n = \int_M (\chi + \delta \omega)^m \wedge \omega^{n - m} .
\end{equation*}
% We normalize $\varphi_1$ and $\varphi_2$ such that
%\begin{equation*}
% \sup_M (\varphi_1 - \varphi_2) = \sup_M (\varphi_2 - \varphi_1) .
%\end{equation*}
 Suppose that for $q > 1$, there is a constant $K > 0$ such that
\begin{equation*}
 \Vert e^{n f_1} \Vert_{L^q} \leq K, \; \Vert e^{n f_2} \Vert_{L^q} \leq K.
\end{equation*}
 Let $\gamma (r)$ be a positive function with $\gamma (r) \to 0 $ as  $r \to 0$. 
 Then there is a constant  $C > 0$ depending on $n, m , \omega , \chi, \tilde\chi, K$ and $\gamma$ such that %for any $r < r_0$, we have
\begin{equation*}
	\inf_{t\in\mathbb{R}}\sup_M |\varphi_1 - \varphi_2 - t| \leq Cr ,
\end{equation*}
 whenever $\Vert e^{m f_1} - e^{ m f_2}\Vert_{L^1} < \gamma(r) r^{ \frac{(n + 1)( n q - m)}{n (q - 1)} + 2} $.

 \end{theorem}

%\begin{remark}
By the stability estimate above, we can find continuous pluripotential solution to complex Hessian equations when the right-sided terms belong to $L^p$ $\left(p > \frac{n}{m}\right)$.
%\end{remark}
\begin{theorem}
\label{theorem-3-11}
Let $\chi \in \bar\Gamma^m_\omega$ be a closed real $(1,1)$-form on a compact K\"ahler manifold $(M,\omega)$ without boundary.
For $q > 1$ and $e^{nf} \in L^q$, there is a  continuous solution in pluripotential sense to 
\begin{equation}
\label{equation-3-63}
(\chi + \delta \omega + \sqrt{-1} \partial\bar\partial\varphi)^m \wedge \omega^{n - m} = e^{mf} \omega^n, \quad \sup_M \varphi = 0.
\end{equation}
\end{theorem}

\begin{proof}

We can find a positive smooth approximation sequence $\left\{e^{m f_i}\right\}$ of $e^{m f}$  with
\begin{equation*}
	\int_M e^{mf_i} \omega^n = \int_M e^{mf} \omega^n
\end{equation*}
and
\begin{equation*}
	\left\Vert e^{m f_i} - e^{m f} \right\Vert_{L^{\frac{nq}{m}}} < \frac{1}{2^{ N ( i + 1)}} ,
\end{equation*}
where $N = \frac{(n + 1) (n q - m)}{n (q - 1)} + 3$.
%We can find a positive smooth approximation $e^{m f_i}$ of $e^{m f} + \frac{1}{2^i}$ with
%$\left\Vert e^{n f_i} - e^{nf} - \frac{1}{2^i}\right\Vert_{L^{\frac{nq}{m}}} < \frac{1}{2^i}$. 
%By normalization,
%\begin{equation}
%	e^{m b_i} \int_M e^{m f_i} \omega^n = \int_M (\chi + \delta\omega)^m \wedge \omega^{n - m} .
%\end{equation}
%By H\"older inequality, we can see that $\int_M e^{m f_i} \omega^n \to \int_M e^{mf} \omega^n$ and thus $b_i \to 0$. Therefore,
%\begin{equation}
%\begin{aligned}
%	\Vert e^{m b_i} e^{m f_i} - e^{mf}\Vert_{L^{\frac{nq}{m}}} 
%	&\leq \Vert e^{m b_i} e^{m f_i} - e^{m b_i} e^{mf} + e^{mb_i} e^{mf} - e^{mf}\Vert_{L^{\frac{nq}{m}}} \\
%	&\leq \Vert e^{m b_i} e^{m f_i} - e^{m b_i} e^{mf}\Vert_{L^{\frac{nq}{m}}} + \Vert e^{mb_i} e^{mf} - e^{mf}\Vert_{L^{\frac{nq}{n}}} \\
%	&\leq e^{m b_i} \Vert  e^{m f_i} -   e^{m f}\Vert_{L^{\frac{nq}{m}}} + \left| e^{m b_i} - 1\right|\Vert e^{mf}\Vert_{L^{\frac{nq}{m}}} \\
%	&\leq  \frac{e^{m b_i}}{2^{i - 1}} + \left| e^{m b_i} - 1\right|\Vert e^{mf}\Vert_{L^{\frac{nq}{m}}} \\
%	&\to 0 .
%\end{aligned}
%\end{equation}
Solving approximation equations
\begin{equation}
	(\chi + \delta \omega + \sqrt{-1} \partial\bar\partial \varphi_i)^m \wedge \omega^{n - m} = e^{m f_i} \omega^n ,\quad \sup_M \varphi_i = 0,
\end{equation}
we have that $\varphi_i$'s are uniformly bounded. 
%we have that $\varphi_i$'s are uniformly bounded by the $L^\infty$ estimate. 
By Theorem~\ref{theorem-3-10}, 
\begin{equation}
	osc (\varphi_i - \varphi_j)
	\leq C \left\Vert  e^{mf_i} - e^{mf_j}\right\Vert^{\frac{1}{N}}_{L^1} 
	\leq C \left\Vert  e^{mf_i} - e^{mf_j}\right\Vert^{\frac{1}{N}}_{L^{\frac{nq}{m}}} < \frac{C}{2^i} .
\end{equation}
%after normalization
%\begin{equation*}
%	\sup_M (\varphi_i - \varphi_j) = \sup_M (\varphi_j - \varphi_i) .
%\end{equation*}
%\begin{equation}
%\begin{aligned}
%	\left\Vert e^{m f_i} - e^{m f_j} \right\Vert_{L^1}
%	&\leq \left\Vert e^{m f_i} - e^{mf_j} \right\Vert_{L^{\frac{nq}{m}}} \left(\int_M \omega^n\right)^{1 - \frac{m}{nq}}
%\end{aligned}
%\end{equation}
%\begin{equation}
%\begin{aligned}
%	\left\Vert e^{m f_i} - e^{m f_j} \right\Vert^{\frac{1}{N}}_{L^1}
%	&\leq \left\Vert e^{m f_i} - e^{mf_j} \right\Vert^{\frac{1}{N}}_{L^{\frac{nq}{m}}} \left(\int_M \omega^n\right)^{\left(1 - \frac{m}{nq}\right) \frac{1}{N}} 
%	\leq \frac{1}{2^{i + 1}} \left(\int_M \omega^n\right)^{\left(1 - \frac{m}{nq}\right) \frac{1}{N}} .
%\end{aligned}
%\end{equation}
Then we have
\begin{equation}
	\varphi_{i + 1} + \frac{C}{2^i} < \varphi_i + \frac{C}{2^{i - 1}}
\end{equation}
and
\begin{equation}
	\varphi_i - \frac{C}{2^{i - 1}} < \varphi_{i + 1} - \frac{C}{2^i} .
\end{equation}
So $\{\varphi_i\}$ is convergent to a function $\varphi$, and 
\begin{equation}
\left|\varphi_i - \varphi \right| \leq \frac{C}{2^i} ,
\end{equation}
which implies that $\varphi$ is continuous. In particular, $\left\{\varphi_i + \frac{C}{2^{i - 1}}\right\}$ is decreasing to $\varphi$, and thus $\varphi$ is a weak solution in pluripotential sense.

\end{proof}

Moreover, the solutions constructed in Theorem~\ref{theorem-3-11} are equicontinuous. We adapt the argument for complex Monge-Amp\`ere equation~\cite{Kolodziej2005} .
\begin{theorem}
Let $\chi \in \bar\Gamma^m_\omega$ be a closed real $(1,1)$-form on a compact K\"ahler manifold $(M,\omega)$ without boundary. 
For $p > \frac{n}{m}$, the set of solutions to complex Hessian equation~\eqref{equation-3-63} with right-handed bounded in $L^{p}$ norm are equicontinuous.
\end{theorem}

\begin{proof}

According to the proof of Theorem~\ref{theorem-3-11}, we can see that the weak solution is a limit of uniformly convergent smooth function sequences. So we only need to show that the set of  solutions are equicontinuous when the right-handed terms are smooth and bounded in $L^p$ norm.

Assume that for some $\epsilon > 0$, there is a sequence $(\varphi_i , z_i,w_i)$ so that
\begin{equation*}
d (z_i , w_i) < \frac{1}{i} ,
\end{equation*}
and
\begin{equation*}
\varphi_i (z_i) - \varphi_i (w_i) > \epsilon .
\end{equation*}
When the right-handed is bounded in $L^p$, the solutions are uniformly bounded. By Lemma~\ref{lemma-3-4} and the compactness of $M$, it is possible to choose a subsequence $(\varphi_{i_j} , z_{i_j}, w_{i_j})$ such that there is a point $z \in M$ and a function  $\varphi$ satisfying that as $j$ approaches $+\infty$,
$\varphi_{i,j} \to \varphi $ in $L^{q^*} $, 
$z_{i_j} \to z $ and $w_{i_j} \to z $ .
By Theorem~\ref{theorem-3-5}, $\varphi_{i_j}$ converges uniformly, and thus $\varphi$ is continuous.
Therefore as $j \to + \infty$,
\begin{equation*}
\begin{aligned}
	\left|\varphi_{i_j} (z_{i_j}) - \varphi_{i_j} (w_{i_j})\right| 
%	&= \left| \varphi_{i_j} (z_{i_j}) - \varphi (z_{i_j}) + \varphi (z_{i_j}) - \varphi(z) + \varphi (z) - \varphi (w_{i_j}) + \varphi (w_{i_j}) - \varphi_{i_j} (w_{i_j})\right| \\
	&\leq \left| \varphi_{i_j} (z_{i_j}) - \varphi (z_{i_j}) \right| + \left| \varphi (z_{i_j}) - \varphi(z)  \right| \\
	&\quad + \left|  \varphi (z) - \varphi (w_{i_j})  \right| + \left| \varphi (w_{i_j}) - \varphi_{i_j} (w_{i_j})\right|  \to 0 ,
\end{aligned}
\end{equation*}
which is a contradiction.

\end{proof}

\medskip

\section{Degenerate Case}

In this section, we shall consider a class of approximation equations of Equation~\eqref{equation-1}: for $0 < t \leq 1$,
\begin{equation}
\label{equation-4-2}
 (\chi + \tilde \chi + t\omega + \sqrt{-1} \partial\bar\partial\varphi_t)^m \wedge \omega^{n - m} = e^{m b_t} e^{mf} \omega^n, \qquad \sup_M \varphi_t = 0,
\end{equation}
where 
\begin{equation*}
\int_M (\chi + \tilde \chi + t \omega)^m\wedge \omega^{n - m} = e^{m b_t}\int_M (\chi + \tilde \chi)^m \wedge \omega^{n - m} .
\end{equation*}
Based on the result of approximation equations, a weak solution in pluripotential sense will be constructed from a decreasing function sequence.

We denote $V_t := \int_M (\tilde \chi + t\omega)^n$ and solve
\begin{equation*}
	(\tilde \chi + t\omega + \sqrt{-1} \partial\bar\partial\tilde \varphi)^n = \frac{V_t}{\int_M \omega^n} \omega^n .
\end{equation*}
By Maclaurin's inequality,
\begin{equation*}
\begin{aligned}
 e^{mb_t} \int_M (\chi + \tilde \chi)^m \wedge \omega^{n - m} 
 &\geq 
% \int_M (\tilde \chi + t\omega)^m \wedge \omega^{n - m} \\
% &= 
 	\int_M (\tilde \chi + t\omega + \sqrt{-1} \partial\bar\partial\tilde\varphi)^m \wedge \omega^{n - m} \\
 &\geq
 	\int_M \left(\frac{(\tilde \chi + t\omega +\sqrt{-1}\partial\bar\partial\tilde \varphi)^n}{\omega^n}\right)^{\frac{m}{n}}\omega^n \\
% 	&= \int_M \left( \frac{V_t}{\int_M \omega^n}\right)^{\frac{m}{n}} \omega^n \\
 	&= {V^{\frac{m}{n}}_t}{\left(\int_M \omega^n\right)^{1 - \frac{m}{n}}}
 	,
\end{aligned}
\end{equation*}
and hence
\begin{equation}
\label{inequality-4-3}
	\frac{\int_M \tilde \chi^n}{\left(\int_M (\chi + \tilde \chi + \omega)^m \wedge \omega^{n - m} \right)^{\frac{n}{m}}}
	\leq
	\frac{V_t}{e^{nb_t}} 
	\leq \frac{ \left(\int_M (\chi + \tilde\chi )^m \wedge \omega^{n - m}\right)^{\frac{n}{m}}}{\left(\int_M \omega^n\right)^{\frac{n - m}{m}}} .
\end{equation}

\medskip

\subsection{$L^\infty$ estimate}

We solve the auxiliary equation
\begin{equation}
\label{equation-4-3}
	(\tilde \chi + t\omega + \sqrt{- 1} \partial\bar\partial \psi_{s,k})^n = \frac{\tau_k (- \varphi_t - s)}{A_{s,k}} e^{n b_t} e^{nf} \omega^n , \quad \sup_M \psi_{s,k} = 0 ,
\end{equation}
where 
$M_s := \{ -\varphi_t - s > 0\}$ 
and
\begin{equation}
	A_{s,k} := \frac{1}{V_t} \int_M  \tau_k (- \varphi_t - s) e^{n b_t} e^{nf} \omega^n \to A_s := \frac{1}{V_t} \int_{M_s}  (- \varphi_t - s) e^{n b_t} e^{nf} \omega^n.
\end{equation}

Define
\begin{equation*}
	\Phi := - \left(\frac{n + 1}{n}\right)^{ \frac{n}{n + 1}} A^{\frac{1}{n + 1}}_{s,k} \left(- \psi_{s,k} + \frac{n}{n + 1} A_{s,k}\right)^{\frac{n}{n + 1}} - \varphi_t - s .
\end{equation*}
As in the previous section, we can show that
\begin{equation}
\label{inequality-4-13}
	- \left(\frac{n + 1}{n}\right)^{\frac{n}{n + 1}} A^{\frac{1}{n + 1}}_{s,k} \left(-\psi_{s,k} + \frac{n}{n + 1} A_{s,k}\right)^{\frac{n}{n + 1}} - \varphi_t - s 
	\leq 0
	\quad \text{on } M
	.
\end{equation}
There are constants $\beta > 0$ and $C > 0$ derived from $\alpha$-invariant so that
\begin{equation*}
\label{inequality-4-14}
\begin{aligned}
	\int_{M_s} \exp \left( \beta  \frac{(- \varphi_t - s)^{\frac{n + 1}{n}}}{A^{\frac{1}{n}}_{s,k}}\right) \omega^n 
	&\leq 
	\int_{M_s} \exp \left( \beta  \left( - \frac{n + 1}{n} \psi_{s,k} +  A_{s,k}\right)\right) \omega^n \\
%	&=
%	\exp \left(\beta A_{s,k}\right) \int_{M_s} \exp \left(- \beta\frac{n + 1}{n} \psi_{s,k}\right) \omega^n \\
	&\leq 
	\exp \left(\beta A_{s,k}\right) \int_M \exp \left(- \beta\frac{n + 1}{n} \psi_{s,k}\right) \omega^n \\
	&\leq C \exp \left(\beta A_{s,k}\right) 
	.
\end{aligned}
\end{equation*}
Letting $k \to \infty$,
\begin{equation*}
	\int_{M_s} \exp \left( \beta  \frac{(- \varphi_t - s)^{\frac{n + 1}{n}}}{A^{\frac{1}{n}}_{s}}\right) \omega^n 
	\leq
	C e^{\beta A_s}
%	\leq 
%	C \exp \left(\frac{\beta}{V_t} \int_{M_s} (- \varphi_t) e^{nb_t} e^{nf} \omega^n\right) 
	.
\end{equation*}
By generalized Young's inequality, we have
\begin{equation}
\label{inequality-4-16}
\begin{aligned}
	&\quad \frac{\beta^p}{2^p A^{\frac{p}{n}}_s} \int_{M_s} (-\varphi_t - s)^{\frac{(n + 1) p}{n}} e^{n f} \omega^n \\
	&\leq 
	\int_{M_s} e^{n f} \left(1 + n|f|\right)^p \omega^n + C \int_{M_s} \exp \left( \beta \frac{(- \varphi_t - s)^{\frac{n + 1}{n}}}{A^{\frac{1}{n}}_s} \right) \omega^n \\
	&\leq 
	\int_{M} e^{n f} \left(1 + n|f|\right)^p \omega^n + C e^{\beta A_s} 
	.	
\end{aligned}
\end{equation}
Then by H\"older inequality with respect to measure $e^{nf} \omega^n$,
\begin{equation}
\label{inequality-4-17}
\begin{aligned}
	A_s
%	&= \frac{1}{V_t} \int_{M_s} (-\varphi_t - s) e^{nb_t} e^{nf} \omega^n \\
%	&= \frac{e^{nb_t}}{V_t} \int_{M_s} (-\varphi_t - s) e^{nf} \omega^n \\
	&\leq 
	\frac{e^{nb_t}}{V_t} \left(\int_{M_s} (-\varphi_t - s)^{\frac{(n + 1) p}{n}} e^{nf} \omega^n\right)^{\frac{n}{(n + 1) p}} \left(\int_{M_s} e^{nf} \omega^n \right)^{\frac{(n + 1) p - n}{(n + 1) p}} \\
	&\leq 
	\frac{e^{nb_t}}{V_t} \frac{2^{\frac{n}{n + 1}}}{\beta^{\frac{n}{n + 1}}} A^{\frac{1}{n + 1}}_s \left(\int_{M} e^{nf} (1 + n|f|)^p \omega^n + C e^{\beta A_s} \right)^{\frac{n}{(n + 1) p}} \left(\int_{M_s} e^{nf} \omega^n\right)^{\frac{(n + 1) p - n}{(n + 1) p}}
	.
\end{aligned}
\end{equation}
The last line in \eqref{inequality-4-17} is from \eqref{inequality-4-16}. Rewriting \eqref{inequality-4-17},
\begin{equation*}
\begin{aligned}
	A_s
%	&\leq 
%	\frac{2 e^{(n + 1) b_t}}{\beta V^{\frac{n + 1}{n}}_t} 
%	\left(\int_{M} e^{nf} (1 + n|f|)^p \omega^n + C e^{\beta A_s} \right)^{\frac{1}{ p}} 
%	\left(\int_{M_s} e^{nf} \omega^n\right)^{1 + \frac{1}{n} - \frac{1}{p}} \\
	&\leq 
	\frac{2 e^{(n + 1) b_t}}{\beta V^{\frac{n + 1}{n}}_t} 
	\left(\int_{M} e^{nf} (1 + n|f|)^p \omega^n + C e^{\beta E} \right)^{\frac{1}{ p}} 
	\left(\int_{M_s} e^{nf} \omega^n\right)^{1 + \frac{1}{n} - \frac{1}{p}} 
	,
\end{aligned}
\end{equation*}
where $E$ is a uniform bound of $A_s$ to be specified later. For any $s', s > 0$,
\begin{equation}
\label{inequality-4-19}
\begin{aligned}
	s' \int_{M_{s + s'}} e^{nf} \omega^n
%	&\leq \int_{M_{s + s'}} (-\varphi_t - s) e^{nf} \omega^n \\
	&\leq \int_{M_{s }} (-\varphi_t - s) e^{nf} \omega^n \\
%	&= \frac{V_t}{e^{nb_t}} A_s  \\
	&\leq   \frac{2 e^{  b_t}}{\beta V^{\frac{ 1}{n}}_t} 
		\left(\int_{M} e^{nf} (1 + n|f|)^p \omega^n + C e^{\beta E} \right)^{\frac{1}{ p}} 
		\left(\int_{M_s} e^{nf} \omega^n\right)^{1 + \frac{1}{n} - \frac{1}{p}} 
		.
\end{aligned}
\end{equation}
Inquality~\eqref{inequality-4-19} is followed by the De Giorgi iteration in Lemma~\ref{lemma-4-1}. 
Therefore, we obtain that $	\int_{M_s} e^{nf} \omega^n = 0$, for any $s \geq 0$ with
\begin{equation}
\label{inequality-4-20}
	 s \geq  
%		\frac{2 e^{  b_t}}{\beta V^{\frac{ 1}{n}}_t} 
%			\left(\int_{M} e^{nf} (1 + n|f|)^p \omega^n + C e^{\beta E} \right)^{\frac{1}{ p}}  
%			\left(\int_M e^{nf} \omega^n\right)^{\frac{1}{n} - \frac{1}{p}}
%			2^{\frac{pn + p - n}{p - n}}
%		= 
		\frac{2^{\frac{pn + 2p - 2n}{p - n}}}{\beta} \left(\frac{e^{mb_t}}{V_t}\right)^{\frac{1}{m}} 
			\left(\int_{M} e^{nf} (1 + n|f|)^p \omega^n + C e^{\beta E} \right)^{\frac{1}{ p}}  
			\left(\int_M e^{nf} \omega^n\right)^{\frac{1}{n} - \frac{1}{p}}
			.
\end{equation}

To obtain the uniform $L^\infty$ estimate of $\varphi_t$, we now only need to control  the upper bound $E$ of $A_s$, which will be obtained through  the argument in \cite{GP202207}.
Since
\begin{equation*}
	A_s = \frac{e^{nb_t}}{V_t} \int_{M_s} (-\varphi_t - s) e^{nf} \omega^n \leq \frac{e^{nb_t}}{V_t} \int_M (-\varphi_t) e^{nf} \omega^n ,
\end{equation*}
it suffices to prove that the bound of $ \int_M (-\varphi_t) e^{nf} \omega^n$ is  independent from parameter $t \in (0,1]$. 
From \eqref{inequality-4-13}, on $M_s$
\begin{equation}
\label{inequality-4-22}
\begin{aligned}
	&\quad (-\varphi_t - s)^{\frac{p(n + 1)}{n}}  e^{nf} \\
	&\leq 
	\left( - \frac{n + 1}{n} \psi_{s,k} +  A_{s,k}\right)^p A^{\frac{p}{n}}_{s,k} e^{nf} \\
	&\leq 
	C \left( (-\psi_{s,k})^p + A^p_{s,k}\right) A^{\frac{p}{n}}_{s,k} e^{nf} \\
	&\leq 
	C \left( e^{nf} \left(1 + n |f|\right)^p + \exp \left({- \beta \frac{n + 1}{n}\psi_{s,k}}\right)\right) A^{\frac{p}{n}}_{s,k}  + C  A^{ \frac{p (n + 1) }{n}}_{s,k} e^{nf} 
	.
\end{aligned}
\end{equation}
%where $\beta$ is the $\alpha$-invariant as defined in \eqref{inequality-4-14}. 
The last line in \eqref{inequality-4-22} is due to generalized Young's inequality.
Integrating \eqref{inequality-4-22} over $M_s$,
\begin{equation}
\label{inequality-4-23}
\begin{aligned}
%	&\quad 
	\int_{M_s} 	(-\varphi_t - s)^{\frac{p(n + 1)}{n}}  e^{nf} \omega^n 
%	\\
%	&\leq 
%	C  A^{\frac{p}{n}}_{s,k} \int_{M_s}\left( e^{nf} \left(1 + n |f|\right)^p + \exp \left(- \beta \frac{n + 1}{n}\psi_{s,k} \right)\right) \omega^n + C  A^{ \frac{p (n + 1)}{n}}_{s,k} \int_{M_s} e^{nf}  \omega^n \\
%	&\leq 
%	C A^{\frac{p}{n}}_{s,k} \int_M e^{nf} (1 + n|f|)^p \omega^n + C A^{\frac{p}{n}}_{s,k}  + C A^{ \frac{p (n + 1)}{n}}_{s,k} \int_{M_s} e^{nf}  \omega^n \\
	&\leq C A^{\frac{p}{n}}_{s,k} + C A^{ \frac{p (n + 1)}{n}}_{s,k} \int_{M_s} e^{nf}  \omega^n 
	.
\end{aligned}
\end{equation}
By H\"older inequality with respect to measure $e^{nf} \omega^n$,
\begin{equation}
\label{inequality-4-24}
\begin{aligned}
	\int_{M_s} (- \varphi_t - s) e^{nf} \omega^n
	&\leq 
	\left(\int_{M_s} (-\varphi_t - s)^{\frac{p (n + 1)}{n}} e^{nf} \omega^n\right)^{\frac{n }{p (n + 1)}} \left(\int_{M_s} e^{nf} \omega^n\right)^{1 - \frac{n }{p (n + 1) }} \\
	&\leq 
	C \left( A^{\frac{p}{n}}_{s,k} +  A^{ \frac{p (n + 1)}{n}}_{s,k} \int_{M_s} e^{nf}  \omega^n \right)^{\frac{n }{p (n + 1)}} \left(\int_{M_s} e^{nf} \omega^n\right)^{1 - \frac{n}{p (n + 1) }} \\
%	&=
%	C A^{\frac{1}{n + 1}}_{s,k} \left(1 + A^p_{s,k} \int_{M_s} e^{nf}  \omega^n \right)^{\frac{n }{p (n + 1)}} \left(\int_{M_s} e^{nf} \omega^n\right)^{1 - \frac{n}{p (n + 1) }} \\
%	&\leq 
%	C A^{\frac{1}{n + 1}}_{s,k} \left(\left(1 + \left(A^p_{s,k} \int_{M_s} e^{nf}  \omega^n\right)^{\frac{n }{p (n + 1)}} \right)^{\frac{p (n + 1)}{n }} \right)^{\frac{n }{p (n + 1)}} \left(\int_{M_s} e^{nf} \omega^n\right)^{1 - \frac{n}{p (n + 1) }} \\
%	&\leq 
%	C A^{\frac{1}{n + 1}}_{s,k}  \left(1 + \left(A^p_{s,k} \int_{M_s} e^{nf}  \omega^n\right)^{\frac{n }{p (n + 1)}} \right)  \left(\int_{M_s} e^{nf} \omega^n\right)^{1 - \frac{n}{p (n + 1) }} \\	
%	&= 
%	C A^{\frac{1}{n + 1}}_{s,k}  \left( \left(\int_{M_s} e^{nf} \omega^n\right)^{1 - \frac{n}{p (n + 1) }}  +  A^{\frac{n}{n + 1}}_{s,k}  \int_{M_s} e^{nf}  \omega^n  \right)   \\
	&\leq 
	C   \left(A^{\frac{1}{n + 1}}_{s,k} \left(\int_{M_s} e^{nf} \omega^n\right)^{1 - \frac{n}{p (n + 1) }}  +  A_{s,k}  \int_{M_s} e^{nf}  \omega^n  \right) 
	,
\end{aligned}
\end{equation}
where the second line in \eqref{inequality-4-24} is due to \eqref{inequality-4-23}.
Letting $k \to \infty$,
%\begin{equation}
%\label{inequality-4-25}
%	\int_{M_s} (- \varphi_t - s) e^{nf} \omega^n
%	\leq 
%	C \left(A^{\frac{1}{n + 1}}_s \left(\int_{M_s} e^{nf} \omega^n\right)^{1 - \frac{n}{p (n + 1) }}  +  \frac{e^{nb_t}}{V_t} \int_{M_s} (-\varphi_t - s) e^{nf} \omega^n  \int_{M_s} e^{nf}  \omega^n  \right) 
%	.
%\end{equation}
\begin{equation}
\label{inequality-4-25}
	\int_{M_s} (- \varphi_t - s) e^{nf} \omega^n
	\leq 
	C \left(A^{\frac{1}{n + 1}}_s \left(\int_{M_s} e^{nf} \omega^n\right)^{1 - \frac{n}{p (n + 1) }}  +  A_s \int_{M_s} e^{nf}  \omega^n  \right) 
	.
\end{equation}

We shall need a lemma from Guo-Phong~\cite{GP202207}.
\begin{lemma}
\label{lemma-4-2}
There exists a constant $C_1 > 0$ such that for any $s > 1$,
\begin{equation*}
	\int_{M_s} e^{nf} \omega^n \leq \frac{C_1}{(\ln s)^p} .
\end{equation*}
\end{lemma}

By Lemma~\ref{lemma-4-2}, we have
\begin{equation}
\label{inequality-4-27}
	\frac{C e^{nb_t}}{V_t} \int_{M_s} e^{nf} \omega^n \leq \frac{C C_1 e^{nb_t}}{V_t (\ln s)^p} < \frac{1}{2},
\end{equation}
if we pick $s\geq s_1 := 1 + \exp\left( \left( \frac{2 C C_1 e^{nb_t}}{V_t} \right)^{\frac{1}{p}} \right)$ where $C$ here is the same as that in \eqref{inequality-4-25}.
Substituting \eqref{inequality-4-27} into \eqref{inequality-4-25},
%\begin{equation}
%	A^{\frac{n}{n + 1}}_s \leq \frac{2 C e^{nb_t}}{V_t}   \left(\int_{M_s} e^{nf} \omega^n\right)^{1 - \frac{n }{p (n + 1)}} .
%\end{equation} 
\begin{equation}
\label{inequality-4-28}
	\int_{M_s} (-\varphi_t - s) e^{nf}\omega^n \leq 2 C\left( \frac{e^{nb_t}}{V_t} \int_{M_s} (-\varphi_t - s) e^{nf} \omega^n  \right)^{\frac{1}{n + 1}} \left(\int_{M_s} e^{nf} \omega^n\right)^{1 - \frac{n }{p (n + 1)}} .
\end{equation}
Rewriting \eqref{inequality-4-28},
\begin{equation}
\label{inequality-4-15-1}
\begin{aligned}
	\int_{M_s} (-\varphi_t - s) e^{nf} \omega^n 
	&\leq C \left(\frac{e^{nb_t}}{V_t}\right)^{\frac{1}{n}}  \left(\int_{M_s} e^{nf} \omega^n\right)^{1 + \frac{1}{n} - \frac{1}{p }}  
%	\\
%	&\leq C  \left(\frac{V_t}{e^{n b_t}}\right)^{1 - \frac{1}{p }} \\
%	&\leq C \frac{ \left(\int_M (\chi + \tilde\chi )^m \wedge \omega^{n - m}\right)^{\frac{n (p - 1)}{m p}}}{\left(\int_M \omega^n\right)^{\frac{(n - m)(p - 1)}{mp}}} \\
%	&=
%	C \frac{\left( e^{mb_0} \int_M e^{mf} \omega^n\right)^{\frac{n (p - 1)}{m p}}}{\left(\int_M \omega^n\right)^{\frac{(n - m)(p - 1)}{mp}}} 
	.
\end{aligned}
\end{equation}
Substituting \eqref{inequality-4-27} and \eqref{inequality-4-3} into \eqref{inequality-4-15-1},
\begin{equation*}
%\begin{aligned}
	\int_{M_s} (-\varphi_t - s) e^{nf} \omega^n 
%	&\leq C \left(\frac{e^{nb_t}}{V_t}\right)^{\frac{1}{n}}  \left(\int_{M_s} e^{nf} \omega^n\right)^{1 + \frac{1}{n} - \frac{1}{p }}  \\
	\leq C  \left(\frac{V_t}{e^{n b_t}}\right)^{1 - \frac{1}{p }} 
	\leq C \frac{ \left(\int_M (\chi + \tilde\chi )^m \wedge \omega^{n - m}\right)^{\frac{n (p - 1)}{m p}}}{\left(\int_M \omega^n\right)^{\frac{(n - m)(p - 1)}{mp}}}  
	.
%\end{aligned}
\end{equation*}
Therefore,
\begin{equation}
\label{inequality-4-30}
\begin{aligned}
	\int_M (-\varphi_t) e^{nf} \omega^n
%	&= \int_M (-\varphi_t - s_1) e^{nf} \omega^n  + s_1 \int_M e^{nf} \omega^n \\
%	&= \int_{M_{s_1}} (-\varphi_t - s_1) e^{nf} \omega^n + \int_{M \setminus M_{s_1}} (-\varphi_t - s_1) e^{nf} \omega^n  + s_1 \int_M e^{nf} \omega^n \\
	&\leq 
	 \int_{M_{s_1}} (-\varphi_t - s_1) e^{nf} \omega^n + s_1\int_{M} e^{nf} \omega^n  \\
	 &\leq 
	C \frac{ \left(\int_M (\chi + \tilde\chi )^m \wedge \omega^{n - m}\right)^{\frac{n (p - 1)}{m p}}}{\left(\int_M \omega^n\right)^{\frac{(n - m)(p - 1)}{mp}}} \\
	&\qquad\qquad
	+ \left(1 + \exp\left( \left( \frac{2 C C_1 e^{nb_t}}{V_t} \right)^{\frac{1}{p}} \right)\right) \int_M e^{nf} \omega^n 
	.
\end{aligned}
\end{equation}

Substituting \eqref{inequality-4-30} into \eqref{inequality-4-20}, we obtain the following $L^\infty$ estimate. 
\begin{theorem}
\label{theorem-4-3}
Suppose that the quantity $\int_M e^{nf} (1 + n|f|)^p$ is bounded for some $p > n$.
Then there is a constant $C$ so that for all $t\in(0,1]$, we have
\begin{equation*}
\Vert\varphi_t\Vert_{L^\infty} < C,
\end{equation*}
for admissible solution $\varphi_t$ to approximation equation~\eqref{equation-4-2}.
\end{theorem}

%
%
%\newpage
%
%
%
%
%Suppose that $e^f$ is bounded in $L^q$ norm with $q > \frac{n^2}{m}$. It is shown in \cite{SuiSun2021} that there is a uniform constant $C > 0$ such that $ -\varphi_t < C$. 
%
%
%

\medskip
\subsection{Stability estimate}

With the $L^\infty$ estimate in Theorem~\ref{theorem-4-3}, we can improve the coefficient $q$ in the strong stability estimate in \cite{GPT202106}\cite{SuiSun2021}. In \cite{SuiSun2021}, the assumption of $q > \frac{n}{m}$ is only used to derive the $L^\infty$ estimate independent of parameter $t$.

\begin{theorem}
\label{theorem-4-4}

Let $\chi \in \bar\Gamma^m_\omega$ be a closed real $(1,1)$-form, and $\tilde \chi$ be a semipositive and big closed $(1,1)$-form on compact K\"ahler manifolds $(M,\omega)$ without boundary. Assume that there are admissible solutions $\varphi_1$ and $\varphi_2$ respectively to complex Hessian equations
\begin{equation*}
	(\chi + \tilde \chi + t\omega + \sqrt{-1} \partial\bar\partial \varphi_1)^m \wedge \omega^{n - m} = e^{mb_t} e^{mf_1} \omega^n ,
\end{equation*}
and
\begin{equation*}
	(\chi + \tilde \chi + t\omega + \sqrt{-1} \partial\bar\partial \varphi_2)^m \wedge \omega^{n - m} = e^{mb_t} e^{mf_2} \omega^n ,
\end{equation*}
where $f_1$ and $f_2$ are real smooth functions on $M$ with
\begin{equation*}
	\int_M e^{mf_1} \omega^n = \int_M e^{mf_2} \omega^n = \int_M (\chi + \tilde \chi)^m \wedge \omega^{n - m} ,
\end{equation*}
and
\begin{equation*}
	\int_M (\chi + \tilde \chi + t\omega)^m \wedge \omega^{n - m} = e^{mb_t} \int_M (\chi + \tilde \chi)^m \wedge \omega^{n - m} .
\end{equation*}
We normalize $\varphi_1$ and $\varphi_2$ so that
\begin{equation*}
	\sup_M (\varphi_1 - \varphi_2) = \sup_M (\varphi_2 - \varphi_1) .
\end{equation*}
Suppose that for $q > 1$, there is a constant $K > 0$ such that
\begin{equation*}
	\Vert e^{nf_1}\Vert_{L^q} \leq K, \quad \Vert e^{nf_2}\Vert_{L^q} \leq K.
\end{equation*}
Let $\gamma (r)$ be a positive function with $\gamma (r) \to 0 $ as  $r \to 0$. 
 Then there is a constant  $C > 0$ depending on $n, m , \omega , \chi, \tilde\chi, K$ and $\gamma$ such that %for any $r < r_0$, we have
\begin{equation*}
	 \sup_M |\varphi_1 - \varphi_2 | \leq Cr ,
\end{equation*}
 whenever $\Vert e^{m f_1} - e^{ m f_2}\Vert_{L^1} < \gamma(r) r^{ \frac{(n + 1)( n q - m)}{n (q - 1)} + 2} $.

\end{theorem}

However, we need another stability estimate in order to construct a solution to Equation~\eqref{equation-1}. 

\begin{theorem}Let $\chi \in \bar\Gamma^m_\omega$ be a closed real $(1,1)$-form, and $\tilde \chi$ be a semipositive and big closed $(1,1)$-form on compact K\"ahler manifolds $(M,\omega)$ without boundary. Assume that we have $C^2$ admissible solutions $\varphi_1$ and $\varphi_2$ respectively to
\begin{equation*}
	(\chi + \tilde \chi + t\omega + \sqrt{-1} \partial\bar\partial \varphi_1 (x,t))^m \wedge \omega^{n - m}
	=
	e^{mb_t} e^{mf_1} \omega^n ,
	\quad
	\sup_M \varphi_1 = 0 ,
\end{equation*}
and
\begin{equation*}
	(\chi + \tilde \chi + t\omega + \sqrt{-1} \partial\bar\partial \varphi_2 (x,t))^m \wedge \omega^{n - m}
	\geq 
	0 ,
	\quad
	\sup_M \varphi_2 \leq 0
	,
\end{equation*}
where
\begin{equation*}
\int_M e^{mf_1} \omega^n  = \int_M (\chi + \tilde \chi)^m \wedge \omega^{n - m} ,
\end{equation*}
and
\begin{equation*}
\int_M (\chi + \tilde \chi + t\omega)^m \wedge \omega^{n - m} = e^{mb_t} \int_M (\chi + \tilde \chi)^m \wedge \omega^{n - m} .
\end{equation*}
Suppose that for $q > 1$, there is a constant $K > 0$ so that
\begin{equation*}
	\Vert e^{nf_1}\Vert_{L^q} \leq K .
\end{equation*}
Then for any $\epsilon > 0$, there exists a constant $C > 0$ such that
\begin{equation}
\label{inequality-4-41}
	\sup_M (\varphi_2 - \varphi_1) 
	\leq 
	C \Vert (\varphi_2 - \varphi_1)^+ \Vert^{\frac{1}{n + 1 + \epsilon}}_{L^{q^*}} ,
\end{equation} 
where $q^* = \frac{q}{q - 1}$.

\end{theorem}

\begin{proof}

By Theorem~\ref{theorem-4-3}, there is a constant $L > 0$ so that
\begin{equation*}
	\Vert \varphi_1 (x,t) \Vert_{L^\infty}\leq L  , \qquad \forall 0 < t \leq 1 .
\end{equation*}
If $\Vert (\varphi_2 - \varphi_1)^+\Vert_{L^{q^*}} = 0$, then
\begin{equation*}
	\varphi_2 \leq \varphi_1 \qquad \text{ on } M , 
\end{equation*}
and Inquality~\eqref{inequality-4-41} holds true. Without loss of generality, we may assume that
\begin{equation*}
	\Vert (\varphi_2 - \varphi_1)^+\Vert_{L^{q^*}} > 0 .
\end{equation*}

We define
\begin{equation*}
	F (\chi) = \left(\frac{\chi^m \wedge \omega^{n - m}}{\omega^n}\right)^{\frac{1}{m}} .
\end{equation*}
It is well known that $F$ is elliptic, homogeneous of degree 1 and concave when $\chi \in \Gamma^m_\omega$. For $0 < r < \frac{1}{2}$ ,  
\begin{equation}
\label{inequality-4-45}
\begin{aligned}
	e^{b_t} e^{f_1}
%	&= F (\chi + \tilde \chi + t\omega + \sqrt{-1} \partial\bar\partial \varphi_1)  \\
	&\geq (1 - r) F(\chi + \tilde \chi + t\omega + \sqrt{-1} \partial\bar\partial\varphi_2) \\
	&\qquad + r F\left(\chi + \tilde \chi + t\omega + \frac{1}{r} \sqrt{-1} \partial\bar\partial \varphi_1 - \frac{1 - r}{r} \sqrt{-1} \partial\bar\partial \varphi_2 \right) \\
	&\geq r F\left( \tilde \chi + t\omega + \frac{1}{r} \sqrt{-1} \partial\bar\partial \varphi_1 - \frac{1 - r}{r} \sqrt{-1} \partial\bar\partial \varphi_2 \right) \\
	&\geq r \left(\frac{\left(\tilde \chi + t\omega +\frac{1}{r} \sqrt{-1} \partial\bar\partial \varphi_1 - \frac{1 - r}{r} \sqrt{-1} \partial\bar\partial \varphi_2\right)^n}{\omega^n}\right)^{\frac{1}{n}}
	,
\end{aligned}
\end{equation}
wherever $\tilde \chi + t\omega + \frac{1}{r} \sqrt{-1} \partial\bar\partial \varphi_1 - \frac{1 - r}{r} \sqrt{-1} \partial\bar\partial \varphi_2 \geq 0$. Rewriting \eqref{inequality-4-45},
\begin{equation*}
	\frac{e^{nb_t} e^{nf_1}}{r^n} \omega^n \geq \left(\tilde \chi + t\omega +\frac{1}{r} \sqrt{-1} \partial\bar\partial \varphi_1 - \frac{1 - r}{r} \sqrt{-1} \partial\bar\partial \varphi_2\right)^n .
\end{equation*}

We can solve the auxiliary equation by \cite{Yau1978},
\begin{equation*}
(\tilde \chi + t\omega + i \partial\bar\partial \psi_k)^n = \frac{\tau_k \left(\frac{-\varphi_1 + (1 - r) \varphi_2}{r} - s\right)}{A_{s,k}} \frac{e^{nb_t} e^{nf_1}}{r^n} \omega^n 	, \qquad \sup_M \psi_k = 0 ,
\end{equation*}
where 
$M_s := \{-\varphi_1 + (1 - r)\varphi_2 > s r\}$ and
%$V_t := \int_M (\tilde \chi + t\omega)^n$ and
\begin{equation*}
\begin{aligned}
A_{s,k} 
&:= 
%\frac{1}{V_t} \int_M \tau_k \left(\frac{-\varphi_1 + (1 - r) \varphi_2}{r} - s\right)  \frac{e^{nb_t} e^{nf_1}}{r^n} \omega^n \\
%&= 
\frac{1}{V_t} \int_M \tau_k \left(\frac{-\varphi_1 + (1 - r) \varphi_2}{r} - s\right)  \frac{e^{nb_t}  e^{nf_1}}{r^n} \omega^n \\
&\to
A_s := \frac{1}{V_t} \int_{M_s}   \left(\frac{-\varphi_1 + (1 - r) \varphi_2}{r} - s\right)  \frac{e^{nb_t} e^{nf_1}}{r^n} \omega^n  
.
\end{aligned}
\end{equation*}
Without loss of generality, we assume that $s \geq L + 1$.
It can be proven as in previous sections that 
%\begin{equation}
%\label{inequality-4-48}
%A^{-\frac{1}{n + 1}}_{s,k} \left( \frac{-\varphi_1 + (1 - r) \varphi_2}{r} - s\right)
% \leq \left(-  \frac{n+1}{n}  \psi_k +  A_{s,k} \right)^{\frac{n}{n + 1}} .
%\end{equation}
%There is a constant $\beta > 0$ such that
%\begin{equation}
%\label{inequality-4-49}
%\begin{aligned}
%%&\quad 
%\int_{M_s} \exp \left(\frac{\beta}{ A^{\frac{1}{n}}_{s,k} }\left( \frac{-\varphi_1 + (1 - r) \varphi_2}{r} - s\right)^{\frac{n+1}{n}} \right) \omega^n
%% \\
%%&\leq 
%%\int_{M_s} \exp \left(- \beta \frac{n+1}{n}  \psi_k + \beta A_{s,k} \right)  \omega^n \\
%&\leq C e^{\beta A_{s,k}} 
%.
%\end{aligned}
%\end{equation}
%Letting $k \to +\infty$ in  \eqref{inequality-4-49}, we obtain
\begin{equation*}
 \int_{M_s} \exp \left( \frac{\beta}{ A^{\frac{1}{n}}_{s} } \left( \frac{-\varphi_1 + (1 - r) \varphi_2}{r} - s\right)^{\frac{n+1}{n}} \right) \omega^n
\leq C e^{\beta A_{s}} .
\end{equation*}
Therefore, 
by H\"older inequality, for $q > 1$ and $p>n$,
\begin{equation*}
\begin{aligned}
&\quad \frac{\beta^p}{A^{\frac{p}{n}}_s} \int_{M_s} \left(\frac{- \varphi_1 + (1 - r) \varphi_2}{r} - s\right)^{\frac{(n + 1)p}{n}} \frac{e^{nb_t} e^{n f_1}}{r^n} \omega^n \\
%&=
%\frac{e^{nb_t}}{r^n} \int_{M_s}\left( \frac{\beta }{A^{\frac{1}{n }}_s}  \left(\frac{- \varphi_1 + (1 - r) \varphi_2}{r} - s\right)^{\frac{n + 1}{n}} \right)^p e^{n f_1}\omega^n 
%\\
&\leq 
\frac{e^{nb_t}}{r^n}
\left(\int_{M_s} \left(\frac{\beta }{A^{\frac{1}{n }}_s}  \left(\frac{- \varphi_1 + (1 - r) \varphi_2}{r} - s\right)^{\frac{n + 1}{n}}\right)^{ \frac{pq}{q - 1}} \omega^n \right)^{1 - \frac{1}{q}}
\left( \int_{M_s} e^{q n f_1} \omega^n\right)^{\frac{1}{q}} \\
%&\leq  
%\frac{e^{n b_t}}{r^n} \left(\int_{M_s} e^{- \frac{pq}{q - 1}} \left(\frac{pq}{q - 1}\right)^{\frac{pq}{q - 1}} \exp \left(\frac{\beta}{A^{\frac{1}{n}}_s} \left(\frac{-\varphi_1 + (1 - r) \varphi_2}{r} - s\right)^{\frac{n + 1}{n}}\right) \omega^n \right)^{1 - \frac{1}{q}}
%\left( \int_{M_s} e^{q n f_1} \omega^n\right)^{\frac{1}{q}} \\
%&= 
%\frac{e^{n b_t}}{r^n} e^{-p} \left(\frac{pq}{q - 1} \right)^p \left( \int_{M_s}\exp \left(\frac{\beta}{A^{\frac{1}{n}}_s} \left(\frac{-\varphi_1 + (1 - r) \varphi_2}{r} - s\right)^{\frac{n + 1}{n}}\right) \omega^n\right)^{1 - \frac{1}{q}}
%\left( \int_{M_s} e^{q n f_1} \omega^n\right)^{\frac{1}{q}} \\
&\leq 
\frac{e^{n b_t}}{r^n} e^{-p} \left(\frac{pq}{q - 1} \right)^p \left( \int_{M_s}\exp \left(\frac{\beta}{A^{\frac{1}{n}}_s} \left(\frac{-\varphi_1 + (1 - r) \varphi_2}{r} - s\right)^{\frac{n + 1}{n}}\right) \omega^n\right)^{1 - \frac{1}{q}}
K \\
%&\leq C  K \frac{e^{nb_t}}{r^n} e^{\beta \left(1 - \frac{1}{q}\right) A_s} \\
&\leq C  K \frac{e^{n b_t}}{r^n} e^{\beta \left(1 - \frac{1}{q}\right) E}
 ,
\end{aligned}
\end{equation*}
where $E$ is a uniform upper bound of $A_s$. Consequently, for any $p > n$,
\begin{equation*}
\begin{aligned}
A_s 
%&=
%\frac{1}{V_t} \int_{M_s} \left( \frac{-\varphi_1 + (1 - r) \varphi_2}{r} - s\right) \frac{e^{nb_t} e^{nf_1}}{r^n} \omega^n \\
&\leq 
\left(\frac{1}{V_t} \int_{M_s} \left(\frac{-\varphi_1 + (1 - r) \varphi_2}{r} - s\right)^{\frac{(n + 1)p}{n}} \frac{e^{n b_t} e^{nf_1}}{r^n} \omega^n\right)^{\frac{n}{(n + 1)p}} \\
&\qquad \cdot \left(\frac{1}{V_t} \int_{M_s} \frac{e^{nb_t} e^{nf_1}}{r^n} \omega^n\right)^{\frac{(n+1)p - n}{(n + 1)p}} \\
&\leq \left( C  K \frac{1}{V_t} \frac{e^{nb_t}}{r^n} e^{\beta (1 - \frac{1}{q}) E} \frac{A^{\frac{p}{n}}_s}{\beta^p} \right)^{\frac{n}{(n + 1)p}} 
\left(\frac{1}{V_t} \int_{M_s} \frac{e^{nb_t} e^{nf_1}}{r^n} \omega^n\right)^{\frac{(n+1)p-n}{(n + 1)p}}  \\
&\leq C (p,q,\beta,K,E) A^{\frac{1}{n + 1}}_s \left(\frac{1}{V_t} \frac{e^{nb_t}}{r^n}\right)^{\frac{n}{(n + 1)p}}
\left(\frac{1}{V_t} \int_{M_s} \frac{e^{nb_t} e^{nf_1}}{r^n} \omega^n\right)^{\frac{(n+1)p-n}{(n+1)p}}  
\end{aligned}
\end{equation*}
and hence
\begin{equation*}
\begin{aligned}
A_s 
&\leq  
%C (p,q,\beta,K,E) \left(\frac{1}{V_t}\frac{e^{nb_t}}{r^n}\right)^{\frac{1}{p}}
%\left(\frac{1}{V_t} \int_{M_s} \frac{e^{nb_t} e^{nf_1}}{r^n} \omega^n\right)^{\frac{(n+1)p-n}{np}}  \\
%&= 
% C (p,q,\beta,K,E) \left( \frac{e^{nb_t}}{V_t r^n}\right)^{\frac{n+1}{n}} 
% \left( \int_{M_s} e^{nf_1} \omega^n\right)^{\frac{(n+1)p-n}{np}} \\
% & = 
 C (p,q,\beta,K,E) \left( \frac{e^{nb_t}}{V_t }\right)^{\frac{n+1}{n}} \frac{1}{r^{n + 1}}
  \left( \int_{M_s} e^{nf_1} \omega^n\right)^{\frac{(n+1)p-n}{np}} 
 .
\end{aligned}
\end{equation*}
%
%
%
%We define
%\begin{equation}
%\phi(s) := \int_{M_s} e^{nf_1} \omega^n .
%\end{equation}
For any $s' \in [0,1)$ and $s \geq L + 1$,
\begin{equation}
\label{inequality-4-54}
\begin{aligned}
%s' \phi(s + s')
%&=
s' \int_{M_{s + s'}} e^{nf_1} \omega^n %\\
&\leq 
\int_{M_{s + s'}} \left( \frac{-\varphi_1 + (1 - r) \varphi_2}{r} - s \right) e^{nf_1} \omega^n \\
&\leq 
\int_{M_{s}} \left( \frac{-\varphi_1 + (1 - r) \varphi_2}{r} - s \right) e^{nf_1} \omega^n \\
%&=
%\frac{r^n V_t}{e^{nb_t}} A_s \\
&\leq 
 C (p,q,\beta,K,E)  \left( \frac{e^{nb_t}}{V_t }\right)^{\frac{1}{n}} \frac{1}{r}
  \left( \int_{M_s} e^{nf_1} \omega^n\right)^{1 + \frac{1}{n} - \frac{1}{p}} %\\
%  &\leq 
%  C_{8} (p,q,\beta,K,E)  \left( \frac{e^{n b_1}}{V_0}\right)^{\frac{1}{n}} \frac{1}{r}
%    \phi^{1 + \frac{1}{n} - \frac{1}{p}}(s)
   .
\end{aligned}
\end{equation}
%The De Giorgi iteration shows that $\int_{M_s} e^{nf_1} \omega^n = 0$ for 
%\begin{equation*}
%s \geq \left(C_{9} (p,q,\beta,K,E,\chi',\chi) \frac{1}{r}\right)^{\frac{np}{p - n}} E + \Vert \varphi_1\Vert_{L^\infty} + \frac{1}{1 - 2^{-\frac{np}{p-n}}} .
%\end{equation*}
Applying Lemma~\ref{lemma-4-1} to Inequality~\ref{inequality-4-54}, we obtain that
%$
%	\int_{M_s} e^{nf_1}\omega^n = 0
%$
%when
%\begin{equation}
%	s
%	\geq 
%	L + C_{9} (p,q,\beta,K,E)  \left( \frac{e^{nb_t}}{V_t }\right)^{\frac{1}{n}} \frac{1}{r} \left( \int_M e^{nf_1} \omega^n\right)^{\frac{1}{n} - \frac{1}{p}}  
%%	\geq 
%%	L + C_{8} (p,q,\beta,K,E)  \left( \frac{e^{nb_t}}{V_t }\right)^{\frac{1}{n}} \frac{1}{r} \left( \int_{M_L} e^{nf_1} \omega^n\right)^{\frac{1}{n} - \frac{1}{p}} 2^{{\frac{p n + p - n}{p - n}} } 
%	.
%\end{equation}
\begin{equation}
\label{inequality-4-55}
\begin{aligned}
&\quad -\varphi_1 + (1 - r)\varphi_2 \\
&\leq 	
(L + 1) r + C (p,q,\beta,K,E)  \left( \frac{e^{nb_t}}{V_t }\right)^{\frac{1}{n}}  \left( \int_{M_{L + 1}} e^{nf_1} \omega^n\right)^{\frac{1}{n} - \frac{1}{p}}  \\
%&\leq 
%(L + 1) r +  C (p,q,\beta,K,E)  \left( \frac{e^{nb_t}}{V_t }\right)^{\frac{1}{n}} 
 %\left(\int_{M_{L + 1}} \left(\frac{- \varphi_1 + (1 - r) \varphi_2}{r} - L\right) e^{nf_1} \omega^n\right)^{\frac{1}{n} - \frac{1}{p}}  \\
%&\leq 
%(L + 1) r +  C (p,q,\beta,K,E)   \left( \frac{e^{nb_t}}{V_t }\right)^{\frac{1}{n}} 
% \left(\int_{M_{L+1}} \left(\frac{(1 - r)(\varphi_2 - \varphi_1)}{r} - \varphi_1 - L\right) e^{nf_1} \omega^n\right)^{\frac{1}{n} - \frac{1}{p}}  \\
&\leq 
(L + 1) r +  C (p,q,\beta,K,E)   \left( \frac{e^{nb_t}}{V_t }\right)^{\frac{1}{n}} 
 \left(\int_{M_{L+1}} \left(\frac{(1 - r)(\varphi_2 - \varphi_1)}{r} \right) e^{nf_1} \omega^n\right)^{\frac{1}{n} - \frac{1}{p}}  \\
%&\leq 
%(L + 1) r +  C (p,q,\beta,K,E)   \left( \frac{e^{nb_t}}{V_t }\right)^{\frac{1}{n}} 
% \left(\int_{M}  \frac{(\varphi_2 - \varphi_1)^+}{r}  e^{nf_1} \omega^n\right)^{\frac{1}{n} - \frac{1}{p}}  \\
&\leq 
(L + 1) r +  \frac{C (p,q,\beta,K,E)}{r^{\frac{1}{n} - \frac{1}{p}}}   \left( \frac{e^{nb_t}}{V_t }\right)^{\frac{1}{n}} 
 \left(\int_{M}  (\varphi_2 - \varphi_1)^+ e^{nf_1} \omega^n\right)^{\frac{1}{n} - \frac{1}{p}}  \\
&\leq 
(L + 1) r +  \frac{C (p,q,\beta,K,E)}{r^{\frac{1}{n} - \frac{1}{p}}}   \left( \frac{e^{nb_t}}{V_t }\right)^{\frac{1}{n}} 
 \left( \Vert (\varphi_2 - \varphi_1)^+ \Vert_{L^{q^*}} \Vert e^{nf_1} \Vert_{L^q}\right)^{\frac{1}{n} - \frac{1}{p}} 
.
\end{aligned}
\end{equation}
Similarly, when $s \geq L + 1$,
\begin{equation}
\label{inequality-4-56}
\begin{aligned}
	A_s 
%	&= 
%	\frac{1}{V_t} \int_{M_s} \left( \frac{-\varphi_1 + (1 - r) \varphi_2}{r} - s\right) \frac{e^{nb_t} e^{nf_1}}{r^n} \omega^n \\
%	&=
%	\frac{1}{V_t} \int_{M_s} \left(  \frac{ (1 - r) }{r}  (\varphi_2 - \varphi_1) - \varphi_1 - s\right) \frac{e^{nb_t} e^{nf_1}}{r^n} \omega^n \\
	&\leq 
	\frac{1}{V_t} \int_{M_s} \frac{ (1 - r) }{r}  (\varphi_2 - \varphi_1)^+\frac{e^{nb_t} e^{nf_1}}{r^n} \omega^n \\
	&\leq \frac{2 e^{n b_1}}{V_0 r^{n + 1}} \int_{M_s} (\varphi_2 - \varphi_1)^+ e^{nf_1} \omega^n \\
%	&\leq  \frac{2 e^{n b_1}}{V_0 r^{n + 1}} \int_{M} (\varphi_2 - \varphi_1)^+ e^{nf_1} \omega^n \\
	&\leq \frac{2 e^{n b_1}}{V_0 r^{n + 1}} \Vert (\varphi_2 - \varphi_1)^+ \Vert_{L^{q^*}} \Vert e^{nf_1} \Vert_{L^q} ,
\end{aligned}
\end{equation}
%where $q^*$ is the conjugate number of $q$. 
that is,
\begin{equation}
\label{inequality-4-57}
	E \leq \frac{2 e^{n b_1}}{V_0 r^{n + 1}} \Vert (\varphi_2 - \varphi_1)^+ \Vert_{L^{q^*}} \Vert e^{nf_1} \Vert_{L^q} .
\end{equation}

We choose $p := \frac{n (n + \epsilon)}{\epsilon}$.
When $\Vert (\varphi_2 - \varphi_1)^+ \Vert_{L^{q^*}} < \frac{1}{2^{n + 1 + \epsilon}}$, we denote
\begin{equation*}
	r := 	\Vert (\varphi_2 - \varphi_1)^+ \Vert^{\frac{1}{n + 1 + \epsilon}}_{L^{q^*}}  < \frac{1}{2}.
\end{equation*}
Then we obtain \eqref{inequality-4-57} and \eqref{inequality-4-55},
\begin{equation*}
	E \leq  \frac{2 e^{n b_1}}{V_0 }  r^{\epsilon}  \Vert e^{nf_1} \Vert_{L^q}  \leq \frac{2^{1 - \epsilon} e^{n b_1}}{V_0}  K ,
\end{equation*}
and 
\begin{equation*}
	- \varphi_1 + (1 - r) \varphi_2 
	\leq 
	\left(L + 1+  C  \left( \frac{e^{nb_t}}{V_t }\right)^{\frac{1}{n}} \right) r .
\end{equation*}
Therefore,
\begin{equation*}
	\varphi_2 - \varphi_1 
	\leq 
	\left( L + 1+  C \left( \frac{e^{nb_t}}{V_t }\right)^{\frac{1}{n}} \right) r .
\end{equation*}
Otherwise when $\Vert (\varphi_2 - \varphi_1)^+ \Vert_{L^{q^*}} \geq \frac{1}{2^{n + 1 + \epsilon}}$,
\begin{equation*}
	\varphi_2 - \varphi_1 \leq L \leq 2^{n + 1 + \epsilon} L \Vert (\varphi_2 - \varphi_1)^+ \Vert_{L^{q^*}}  .
\end{equation*}

\end{proof}

\medskip
\subsection{Weak solution}
By Lemma~\ref{lemma-3-4}, there is a sequence $t_i \to 0$ such that $\varphi_{t_i} $ is convergent in $L^{q^*}$ norm.

Observing the simple fact that $\chi + \tilde \chi + t\omega + \sqrt{-1} \partial\bar\partial \varphi_{t'} \in \Gamma^m_{\omega}$ for $0 < t' < t$, we can pick a subsequence $\{\varphi_{t_i}\}$ such that 
\begin{equation*}
\sup_M (\varphi_{t_{i+1}} - \varphi_{t_i} ) < \frac{1}{2^{i + 1}} .
\end{equation*}
Then, $\varphi_{t_i} + \frac{1}{2^i}$ is decreasing to a bounded function $\varphi$, and also $\varphi_{t_i}$ is convergent to $\varphi$. 
Therefore, $\varphi$ is a solution in pluripotential sense.

If $e^{nf}$ is simply a $L^q$ function, we can apply the classical diagonal method to choose a smooth sequence decreasing to certain bounded function $\varphi$, with the help of Theorem~\ref{theorem-4-4}. Therefore, we obtain the weak solution in pluripotential sense.

\begin{theorem}
Let $\chi \in \bar\Gamma^m_\omega$ be a closed real $(1,1)$-form, and $\tilde \chi$ be a semipositive and big closed $(1,1)$-form on compact K\"ahler manifolds $(M,\omega)$ without boundary. 
Then there is a  bounded solution in pluripotential sense to Equation~\eqref{equation-1} 
for $q > 1$ and $e^{nf} \in L^q$. 

\end{theorem}

Noticing that $\varphi_{t_i}$ is bounded in $W^{1,2}$, we can pick a subsequence weakly convergent to a certain function $\tilde \varphi \in W^{1,2}$. Passing to a subsequence again, we may also assume that $\varphi_{t_i}$ is convergent to $\tilde\varphi$ almost everywhere, which implies that $\tilde \varphi = \varphi$. Moreover,  a similar argument to \cite{Cegrell2007}  shows that $\varphi_{t_i}$ is convergent in $W^{1,2}$ to $\varphi$.

\medskip

\section{Uniqueness}

In this section, we shall show that the admissible solution to complex Hessian equation~\eqref{equation-1} is unique up to an additive constant under appropriate conditions.

\begin{theorem}
Let $\chi \in \bar\Gamma^m_\omega$ be a closed real $(1,1)$-form, and $\tilde \chi$ be a semipositive and big closed $(1,1)$-form on compact K\"ahler manifolds $(M,\omega)$ without boundary in this section. Assume that under pluripotential sense
\begin{equation}
\label{equation-5-1}
	(\chi + \tilde \chi + \sqrt{-1} \partial\bar\partial \varphi_1)^m \wedge \omega^{n - m} = e^{mf}  \omega^n, \quad       ess \sup_M \varphi_1 = 0
\end{equation}
and
\begin{equation}
\label{equation-5-2}
	(\chi + \tilde \chi + \sqrt{-1} \partial\bar\partial \varphi_2)^m \wedge \omega^{n - m} = e^{mf}  \omega^n, \quad      ess  \sup_M \varphi_2 = 0 . 
\end{equation}
%where $e^{nf} \in L^q$ for $q > 1$. 
Then $\varphi_1 = \varphi_2$ almost everywhere. 
\end{theorem}

\begin{proof}

%We shall show that $\varphi_1 = \varphi_2$ almost everywhere up to an additive constant. 
%Observing that $\tilde \chi$ can be replaced by any semipositive $(1,1)$-form in class $[\tilde\chi]$ in the following argument, we may claim that $\varphi_1 = \varphi_2$ in $Amp(\tilde\chi)$, if it holds true that $\varphi_1 = \varphi_2$ outside the vanishing locus of $\tilde \chi^n$.
By the work of Eyssidieux-Guedj-Zeriahi~\cite{Eyssidieux-Guedj-Zeriahi}, there is a bounded solution to 
\begin{equation*}
	(\tilde \chi + \sqrt{-1} \partial\bar\partial\psi)^n = \frac{\int_M \tilde \chi^n}{\int_M \omega^n} \omega^n, \qquad \sup_M \psi = 0,
\end{equation*}
under pluripotential sense.  In particular, $\psi$ is smooth on $Amp(\tilde \chi)$, whose the complement  is codimensional at least $2$. For convenience, we still use $\tilde \chi$, $\varphi_1$ and $\varphi_2$ to denote $\tilde \chi + \sqrt{-1} \partial\bar\partial \psi$, $\varphi_1 - \psi$ and $\varphi_2 - \psi$ respectively.

The difference of \eqref{equation-5-1} and \eqref{equation-5-2} can be written as
%\begin{equation}
%\begin{aligned}
%	0 
%	&= \int_M  (\varphi_2 - \varphi_1) \\
%	&\qquad\qquad \cdot\left(	(\chi + \tilde \chi + \sqrt{-1} \partial\bar\partial \varphi_1)^m \wedge \omega^{n - m} - 	(\chi + \tilde \chi +  \sqrt{-1} \partial\bar\partial \varphi_2)^m \wedge \omega^{n - m}\right) \\
%	&= \int_M (\varphi_2 - \varphi_1)  \sqrt{-1} \partial\bar\partial (\varphi_1 - \varphi_2) \\
%	&\qquad\qquad \wedge \sum^{m - 1}_{k = 0} 	(\chi + \tilde \chi +  \sqrt{-1} \partial\bar\partial \varphi_1)^k \wedge 	(\chi + \tilde \chi +  \sqrt{-1} \partial\bar\partial \varphi_2)^{m - 1 - k}  \wedge \omega^{n - m} \\
%	&= \int_M  \sqrt{-1} \partial (\varphi_1 - \varphi_2) \wedge \bar\partial (\varphi_1 - \varphi_2) \\
%	&\qquad\qquad \wedge \sum^{m - 1}_{k = 0}  (\chi + \tilde \chi +  \sqrt{-1} \partial\bar\partial \varphi_1)^k \wedge 	(\chi + \tilde \chi +  \sqrt{-1} \partial\bar\partial \varphi_2)^{m - 1 - k}  \wedge \omega^{n - m} .
%\end{aligned}
%\end{equation}
\begin{equation}
\label{equality-5-3}
\begin{aligned}
	0 
%	&= \int_M  (\varphi_2 - \varphi_1) \\
%	&\qquad\qquad \cdot\left(	(\chi + \tilde \chi +  \sqrt{-1} \partial\bar\partial \varphi_1)^m \wedge \omega^{n - m} - 	(\chi + \tilde \chi +  \sqrt{-1} \partial\bar\partial \varphi_2)^m \wedge \omega^{n - m}\right) \\
	&= \int_M  (\varphi_2 - \varphi_1)  	(\chi + \tilde \chi +  \sqrt{-1} \partial\bar\partial (\theta \varphi_1 + (1 - \theta) \varphi_2))^m \wedge \omega^{n - m} \bigg|^1_{\theta = 0} %\\
%	&= m \int^1_0  \bigg(\int_M (\varphi_2 - \varphi_1)    \sqrt{-1}  \partial\bar\partial (\varphi_1 - \varphi_2) \wedge  (\chi + \tilde \chi +  \sqrt{-1}  \partial\bar\partial (\theta \varphi_1 + (1 - \theta) \varphi_2))^{m - 1} \wedge \omega^{n - m}\bigg) d \theta 
	.
\end{aligned}
\end{equation}
We can pointwisely express
\begin{equation}
\label{equality-5-4}
\begin{aligned}
	&\quad 	(\chi + \tilde \chi +  \sqrt{-1} \partial\bar\partial (\theta \varphi_1 + (1 - \theta) \varphi_2))^m \wedge \omega^{n - m} \bigg|^1_{\theta = 0} \\
	&= m \int^1_0 \left(\sqrt{-1} \partial\bar\partial (\varphi_1 - \varphi_2) \wedge \left(\chi + \tilde \chi + \sqrt{-1} \partial\bar\partial (\theta \varphi_1 + (1 - \theta) \varphi_2)\right)^{m - 1} \wedge \omega^{n - m} \right) d\theta
	.
\end{aligned}
\end{equation}
Equality~\eqref{equality-5-4} is well defined, since the integrand in the right-hand side is just a polynomial of $\theta$. The result has no difference from a direct expansion, but is convenient in the following calculation. Substituting \eqref{equality-5-4} into \eqref{equality-5-3} and applying Stokes' Lemma,
\begin{equation}
\label{equality-5-5}
\begin{aligned}
	0 
%	&= \int_M  (\varphi_2 - \varphi_1) \\
%	&\qquad\qquad \cdot\left(	(\chi + \tilde \chi +  \sqrt{-1} \partial\bar\partial \varphi_1)^m \wedge \omega^{n - m} - 	(\chi + \tilde \chi +  \sqrt{-1} \partial\bar\partial \varphi_2)^m \wedge \omega^{n - m}\right) \\
%	&= \int_M  (\varphi_2 - \varphi_1)  	(\chi + \tilde \chi +  \sqrt{-1} \partial\bar\partial (\theta \varphi_1 + (1 - \theta) \varphi_2))^m \wedge \omega^{n - m} \bigg|^1_{\theta = 0} \\
%	&= m \int^1_0  \bigg(\int_M (\varphi_2 - \varphi_1)    \sqrt{-1}  \partial\bar\partial (\varphi_1 - \varphi_2) \\
%	&\qquad \qquad \qquad \quad \wedge  (\chi + \tilde \chi +  \sqrt{-1}  \partial\bar\partial (\theta \varphi_1 + (1 - \theta) \varphi_2))^{m - 1} \wedge \omega^{n - m}\bigg) d \theta \\
	&= m \int^1_0  \bigg(\int_M  \sqrt{-1}  \partial(\varphi_1 - \varphi_2) \wedge \bar\partial (\varphi_1 - \varphi_2) \\
	&\qquad \qquad \qquad \quad \wedge (\chi + \tilde \chi +  \sqrt{-1}  \partial\bar\partial (\theta \varphi_1 + (1 - \theta) \varphi_2))^{m - 1} \wedge \omega^{n - m}\bigg) d \theta 
	.
\end{aligned}
\end{equation}
If $m = 1$, then we obtain 
\begin{equation*}
	0 = m \int_M \sqrt{-1} \partial (\varphi_1 - \varphi_2) \wedge \bar\partial (\varphi_1 - \varphi_2) \wedge \omega^{n - m} ,
\end{equation*}
which implies that $\varphi_1 - \varphi_2 \equiv 0$ on $M$ up to an additive constant. Without loss of generality, we may assume that $m \geq 2$ in the following argument.

From  \eqref{equality-5-5}, we know that almost everywhere on $M\times [0,1]$,
%\begin{equation}
%	\frac{\sqrt{-1}  \partial(\varphi_1 - \varphi_2) \wedge \bar\partial (\varphi_1 - \varphi_2) \wedge (\chi + \tilde \chi +  \sqrt{-1}  \partial\bar\partial (\theta \varphi_1 + (1 - \theta) \varphi_2))^{m - 1} \wedge \omega^{n - m}}{\omega^n}
%	=
%	0
%	.
%\end{equation}
\begin{equation}
\label{equality-5-6}
\begin{aligned}
	&
	\sqrt{-1}  \partial(\varphi_1 - \varphi_2) \wedge \bar\partial (\varphi_1 - \varphi_2) \\
	&\qquad \qquad \wedge (\chi + \tilde \chi +  \sqrt{-1}  \partial\bar\partial (\theta \varphi_1 + (1 - \theta) \varphi_2))^{m - 1} \wedge \omega^{n - m} 
	= 0
	.
\end{aligned}
\end{equation}
By concavity of $S_k$, for any $0 \leq \theta_1, \theta_2 \leq 1$,
\begin{equation*}
\begin{aligned}
	&\quad \frac{1}{2^{m - 1}}\sqrt{-1} \partial (\varphi_1 - \varphi_2) \wedge \bar\partial (\varphi_1 - \varphi_2) \\
	&\quad \qquad\qquad \wedge (\chi + \tilde \chi + \sqrt{-1} \partial\bar\partial (\theta_1 \varphi_1 + (1 - \theta_1) \varphi_2))^{m - 1} \wedge \omega^{n - m} \\
	&\quad + \frac{1}{2^{m - 1}}\sqrt{-1} \partial (\varphi_1 - \varphi_2) \wedge \bar\partial (\varphi_1 - \varphi_2) \\
		&\qquad \qquad\qquad \wedge (\chi + \tilde \chi + \sqrt{-1} \partial\bar\partial (\theta_2 \varphi_1 + (1 - \theta_2) \varphi_2))^{m - 1} \wedge \omega^{n - m} \\
	&\leq \sqrt{- 1} \partial (\varphi_1 - \varphi_2) \wedge \bar\partial (\varphi_1 - \varphi_2) \\
	&\qquad \qquad \qquad \wedge\left(\chi + \tilde \chi + \sqrt{-1} \partial\bar\partial \left( \frac{\theta_1 + \theta_2}{2} \varphi_1 + \left(1 - \frac{\theta_1 + \theta_2}{2}\right) \varphi_2\right)\right)^{m - 1} \wedge \omega^{n - m} ,
\end{aligned}
\end{equation*}
and hence we can conclude from \eqref{equality-5-6} that fixing $\theta \in [0,1]$, 
\begin{equation}
\label{equality-5-7}
\begin{aligned}
	&
	\sqrt{-1}  \partial(\varphi_1 - \varphi_2) \wedge \bar\partial (\varphi_1 - \varphi_2) \\
	&\qquad \qquad \wedge (\chi + \tilde \chi +  \sqrt{-1}  \partial\bar\partial (\theta \varphi_1 + (1 - \theta) \varphi_2))^{m - 1} \wedge \omega^{n - m} 
	= 0
\end{aligned}
\end{equation}
almost everywhere on $M$. 
Applying Maclaurin's inequality to \eqref{equality-5-7}, for $2\leq k \leq m$
\begin{equation}
\label{equality-5-8}
\begin{aligned}
	&
	\sqrt{-1}  \partial(\varphi_1 - \varphi_2) \wedge \bar\partial (\varphi_1 - \varphi_2) \\
	&\qquad \qquad \wedge (\chi + \tilde \chi +  \sqrt{-1}  \partial\bar\partial (\theta \varphi_1 + (1 - \theta) \varphi_2))^{k - 1} \wedge \omega^{n - k} 
	= 0
	,
\end{aligned}
\end{equation}
almost everywhere on $M$ for $\theta \in [0,1]$. 
In particular,
\begin{equation}
\label{equality-5-9}
\begin{aligned}
	&
	\sqrt{-1} \partial(\varphi_1 - \varphi_2) \wedge \bar\partial (\varphi_1 - \varphi_2) \\
		&\qquad \qquad \wedge (\chi + \tilde \chi +  \sqrt{-1}  \partial\bar\partial (\theta \varphi_1 + (1 - \theta) \varphi_2)) \wedge \omega^{n - 2} 
		= 0 .
\end{aligned}
\end{equation}
%Integrating \eqref{equality-5-8} over $M \times [0,1]$,
%\begin{equation*}
%\begin{aligned}
%	0
%	&=
%	2 \int^1_0 \bigg(\int_M \sqrt{-1} \partial (\varphi_1 - \varphi_2) \wedge \bar\partial (\varphi_1 - \varphi_2) \\
%	&\qquad \qquad \qquad \qquad  \wedge (\chi + \tilde \chi + \sqrt{-1} \partial\bar\partial (\theta \varphi_1 + (1 - \theta ) \varphi_2))  \wedge \omega^{n - 2} \bigg) d\theta \\
%	&=  \int_M \sqrt{-1} \partial (\varphi_1 - \varphi_2) \wedge \bar\partial (\varphi_1 - \varphi_2)  \wedge (\chi + \tilde \chi + \sqrt{-1} \partial\bar\partial \varphi_1) \wedge \omega^{n - 2} \\
%	&\qquad + \int_M \sqrt{-1} \partial (\varphi_1 - \varphi_2) \wedge \bar\partial (\varphi_1 - \varphi_2)  \wedge  (\chi + \tilde \chi + \sqrt{-1} \partial\bar\partial\varphi_2) \wedge \omega^{n - 2} ,
%\end{aligned}
%\end{equation*}
Integrating \eqref{equality-5-9} over $M \times [0,1]$,
\begin{equation}
\label{equality-5-10}
\begin{aligned}
	0
	&=
	\int^1_0 \Bigg(\int_M 	\sqrt{-1} \partial(\varphi_1 - \varphi_2) \wedge \bar\partial (\varphi_1 - \varphi_2) \\
			&\qquad \qquad \qquad \qquad \wedge (\chi + \tilde \chi +  \sqrt{-1}  \partial\bar\partial (\theta \varphi_1 + (1 - \theta) \varphi_2)) \wedge \omega^{n - 2} \Bigg) d\theta\\
%	&= 
%	\int^1_0 \Bigg(\int_M 	\sqrt{-1} \partial(\varphi_1 - \varphi_2) \wedge \bar\partial (\varphi_1 - \varphi_2) \wedge (\chi + \tilde \chi) \wedge \omega^{n - 2} \Bigg) d\theta \\
%	&\qquad + 	\int^1_0 \Bigg(\int_M 	\sqrt{-1} \partial(\varphi_1 - \varphi_2) \wedge \bar\partial (\varphi_1 - \varphi_2) \\
%				&\qquad \qquad \qquad \qquad \quad \wedge   \sqrt{-1}  \partial\bar\partial (\theta \varphi_1 + (1 - \theta) \varphi_2) \wedge \omega^{n - 2} \Bigg) d\theta \\
	&= 
	 \int_M 	\sqrt{-1} \partial(\varphi_1 - \varphi_2) \wedge \bar\partial (\varphi_1 - \varphi_2) \wedge (\chi + \tilde \chi) \wedge \omega^{n - 2}  \\
	&\qquad + \frac{1}{2}	\int_M 	\sqrt{-1} \partial(\varphi_1 - \varphi_2) \wedge \bar\partial (\varphi_1 - \varphi_2)  \wedge   \sqrt{-1}  \partial\bar\partial ( \varphi_1 +   \varphi_2) \wedge \omega^{n - 2}  
	.
\end{aligned}
\end{equation}
Apply Stokes' Lemma,
\begin{equation}
\label{equality-5-11}
\begin{aligned}
	&\quad \int_M 	\sqrt{-1} \partial(\varphi_1 - \varphi_2) \wedge \bar\partial (\varphi_1 - \varphi_2)  \wedge   \sqrt{-1}  \partial\bar\partial ( \varphi_1 +   \varphi_2) \wedge \omega^{n - 2} \\
	&= \int_M 	\sqrt{-1} \partial(\varphi_1 - \varphi_2) \wedge \sqrt{-1}  \partial\bar\partial (\varphi_1 - \varphi_2)  \wedge   \bar\partial ( \varphi_1 +   \varphi_2) \wedge \omega^{n - 2} \\
	&= \int_M 	\sqrt{-1} \partial(\varphi_1 - \varphi_2)   \wedge   \bar\partial ( \varphi_1 +   \varphi_2) \wedge (\chi + \tilde \chi + \sqrt{-1} \partial\bar\partial \varphi_1) \wedge \omega^{n - 2} \\
	&\qquad - \int_M 	\sqrt{-1} \partial(\varphi_1 - \varphi_2)   \wedge   \bar\partial ( \varphi_1 +   \varphi_2) \wedge (\chi + \tilde \chi + \sqrt{-1} \partial\bar\partial \varphi_2) \wedge \omega^{n - 2} .
\end{aligned}
\end{equation}
By Cauchy-Schwarz inequality,
\begin{equation}
\label{inequality-5-12}
\begin{aligned}
	&\quad \left(\int_M 	\sqrt{-1} \partial(\varphi_1 - \varphi_2)   \wedge   \bar\partial ( \varphi_1 +   \varphi_2) \wedge (\chi + \tilde \chi + \sqrt{-1} \partial\bar\partial \varphi_1) \wedge \omega^{n - 2} \right)^2 \\
	&\leq \int_M 	\sqrt{-1} \partial(\varphi_1 - \varphi_2)   \wedge   \bar\partial ( \varphi_1 - \varphi_2) \wedge (\chi + \tilde \chi + \sqrt{-1} \partial\bar\partial \varphi_1) \wedge \omega^{n - 2} \\
	&\qquad \cdot \int_M 	\sqrt{-1} \partial(\varphi_1 + \varphi_2)   \wedge   \bar\partial ( \varphi_1 + \varphi_2) \wedge (\chi + \tilde \chi + \sqrt{-1} \partial\bar\partial \varphi_1) \wedge \omega^{n - 2} .
\end{aligned}
\end{equation}
By Stokes' Lemma again,
\begin{equation}
\label{inequality-5-13}
\begin{aligned}	
	0 
	&\leq \int_M 	\sqrt{-1} \partial(\varphi_1 + \varphi_2)   \wedge   \bar\partial ( \varphi_1 + \varphi_2) \wedge (\chi + \tilde \chi + \sqrt{-1} \partial\bar\partial \varphi_1) \wedge \omega^{n - 2} \\
	&= - \int_M (\varphi_1 + \varphi_2)   	\sqrt{-1} \partial \bar\partial ( \varphi_1 + \varphi_2) \wedge (\chi + \tilde \chi + \sqrt{-1} \partial\bar\partial \varphi_1) \wedge \omega^{n - 2} \\
	&= - 2 \int_M (\varphi_1 + \varphi_2)  \left( \chi + \tilde \chi +	\sqrt{-1} \partial \bar\partial \left(\frac{ \varphi_1 + \varphi_2}{2} \right)\right) \wedge (\chi + \tilde \chi + \sqrt{-1} \partial\bar\partial \varphi_1) \wedge \omega^{n - 2} \\
	&\qquad + 2 \int_M (\varphi_1 + \varphi_2)  \left( \chi + \tilde \chi  \right) \wedge (\chi + \tilde \chi + \sqrt{-1} \partial\bar\partial \varphi_1) \wedge \omega^{n - 2} \\
	&\leq 2 (\Vert\varphi_1\Vert_{L^\infty} + \Vert\varphi_2 \Vert_{L^\infty}) \int_M (\chi + \tilde \chi)^2 \wedge \omega^{n - 2}
	.
\end{aligned}
\end{equation}
Substituting \eqref{inequality-5-13} and \eqref{equality-5-9} into \eqref{inequality-5-12},
\begin{equation}
\label{equality-5-14}
	\int_M 	\sqrt{-1} \partial(\varphi_1 - \varphi_2)   \wedge   \bar\partial ( \varphi_1 +   \varphi_2) \wedge (\chi + \tilde \chi + \sqrt{-1} \partial\bar\partial \varphi_1) \wedge \omega^{n - 2} 
	= 0
	.
\end{equation}
Similarly, we also obtain
\begin{equation}
\label{equality-5-15}
	\int_M 	\sqrt{-1} \partial(\varphi_1 - \varphi_2)   \wedge   \bar\partial ( \varphi_1 +   \varphi_2) \wedge (\chi + \tilde \chi + \sqrt{-1} \partial\bar\partial \varphi_2) \wedge \omega^{n - 2} 
	= 0
	.
\end{equation}
Substituting \eqref{equality-5-14} and \eqref{equality-5-15} into \eqref{equality-5-11},
\begin{equation}
	\int_M \sqrt{-1} \partial (\varphi_1 - \varphi_2) \wedge \bar\partial (\varphi_1 - \varphi_2) \wedge \sqrt{-1} \partial\bar\partial (\varphi_1 + \varphi_2) \wedge \omega^{n - 2} 
	=
	0
	,
\end{equation}
and hence from \eqref{equality-5-10}, 
\begin{equation}
	\int_M \sqrt{-1} \partial (\varphi_1 - \varphi_2) \wedge \bar\partial (\varphi_1 - \varphi_2) \wedge (\chi + \tilde \chi) \wedge \omega^{n - 2} = 0 .
\end{equation}
%Therefore $\varphi_1 - \varphi_2$ must be constant in the ample locus.

\end{proof}

\medskip
\section{Uniform smoothness in case of nonnegative bisectional curvature}

%
%\begin{equation}
%	(\chi + \tilde \chi + t\omega + \sqrt{-1} \partial\bar\partial \varphi_t)^m \wedge \omega^{n - m} = e^{m b_t} e^{m f} \omega^n ,
%\end{equation}
%where $e^{ f}$ is smooth.

Suppose that $e^{ f}$ is smooth. 
Rewriting Equation~\eqref{equation-4-2},
\begin{equation}
\label{equation-6-1}
	S_m (X) 
	= 
	F_t 
	:= 
	C^m_n e^{m b_t} e^{m f} .
\end{equation}
Differentiating \eqref{equation-6-1}
\begin{equation}
\label{derivative-6-2}
	F_{t,l} = \sum_i S_{m - 1;i} X_{i\bar il}
\end{equation}
and
\begin{equation}
\label{derivative-6-3}
	F_{t, l\bar l} = - \sum_{i\neq j} S_{m - 2;ij} X_{i\bar jl} X_{j\bar i\bar l} + \sum_{i\neq j} S_{m - 2;ij} X_{j\bar jl} X_{i\bar i\bar l} + \sum_i S_{m - 1;i} X_{i\bar il\bar l}
\end{equation}
where
\begin{equation*}
	X_{i\bar jk} = \nabla X \left(\frac{\partial}{\partial z^i} , \frac{\partial}{\partial\bar z^j} , \frac{\partial}{\partial z^k} \right) , 
	X_{i\bar j k\bar l} = \nabla^2 X \left(\frac{\partial}{\partial z^i} , \frac{\partial}{\partial\bar z^j} , \frac{\partial}{\partial z^k} , \frac{\partial}{\partial\bar z^l} \right)  , 
	\cdots 
\end{equation*}
Since $S^{\frac{1}{m}}_m$ is concave,
\begin{equation*}
\begin{aligned}
	\sum_{i\neq j} S_{m - 2;ij} X_{j\bar jl} X_{i\bar i\bar l} 
	&\leq \frac{m - 1}{m S_m} \sum_{i,j} S_{m - 1;i} S_{m - 1;j} X_{j\bar jl} X_{i\bar i\bar l} 
	= \frac{m - 1}{m} \frac{F_{t,l} F_{t,\bar l}}{F_t}
%	\\
%	&= \left(1 - \frac{1}{m}\right) C^m_n e^{m b_t} \frac{e^{2 (m - 1)f} \partial_l e^f \bar\partial_l e^f}{e^{mf}} \\
%	&= \left(1 - \frac{1}{m}\right) C^m_n e^{m b_t}  e^{(m - 2)f} \partial_l e^f \bar\partial_l e^f 
	,
\end{aligned}
\end{equation*}
and hence from \eqref{derivative-6-2} and \eqref{derivative-6-3}, 
%\begin{equation}
%	F_{t, l\bar l} 
%%	= - \sum_{i\neq j} S_{m - 2;ij} X_{i\bar jl} X_{j\bar i\bar l} + \sum_{i\neq j} S_{m - 2;ij} X_{j\bar jl} X_{i\bar i\bar l} + \sum_i S_{m - 1;i} X_{i\bar il\bar l}
%	\leq - \sum_{i\neq j} S_{m - 2;ij} X_{i\bar jl} X_{j\bar i\bar l} +  \left(1 - \frac{1}{m}\right) \frac{F_{t,l} F_{t,\bar l}}{F_t} + \sum_i S_{m - 1;i} X_{i\bar il\bar l}
%	.
%\end{equation}
\begin{equation}
\label{inequality-6-4}
	F_{t, l\bar l}  - \frac{m - 1}{m}\frac{F_{t,l} F_{t,\bar l}}{F_t} +  \sum_{i\neq j} S_{m - 2;ij} X_{i\bar jl} X_{j\bar i\bar l} 
	\leq  \sum_i S_{m - 1;i} X_{i\bar il\bar l}
	.
\end{equation}

In the work of Boucksom~\cite{Boucksom2004}, %Demailly-Paun~\cite{Demailly-Paun},
 it is proven that there is a function $\rho$ and a constant $\kappa > 0$ such that $\rho$ is smooth in the ample locus $Amp (\tilde\chi)$ with analytic singularities, $\sup_M \rho = 0$
and
%\begin{equation*}
$\tilde\chi + \sqrt{-1}\partial\bar\partial\rho \geq \kappa \omega $.
%\end{equation*}
Then, we have
\begin{equation}
\label{inequality-6-5}
\begin{aligned}
	\sum_i S_{m - 1;i} \left(\rho_{i\bar i} - \varphi_{t,i\bar i}\right) 
%	&=
%	\sum_i S_{m - 1;i} \left(\chi_{i\bar i} + \tilde \chi_{i\bar i} + t\omega_{i\bar i} + \rho_{i\bar i}\right) - \sum_i S_{m - 1;i} \left(\chi_{i\bar i} + \tilde \chi_{i\bar i} + t\omega_{i\bar i} + \varphi_{t,i\bar i}\right) \\
	&\geq
	\sum_i S_{m - 1;i} \left(\chi_{i\bar i} + \kappa \omega_{i\bar i} + t\omega_{i\bar i}  \right) - \sum_i S_{m - 1;i} X_{i\bar i} \\
	&\geq \kappa \sum_i S_{m - 1;i} - m F_t 
%	\\
%	&\geq \kappa \sum_i S_{m - 1;i} - m F_1 
	.
\end{aligned}
\end{equation}

We shall consider function
$ 	w + A(\rho - \varphi_t) $, 
where $A$ is to be specified later and 
\begin{equation*}
w 
:= tr_\omega (\chi + \tilde \chi + t\omega + \sqrt{-1} \partial\bar\partial \varphi_t) = S_1 (X) .
\end{equation*}
Since $\rho$ tends to $-\infty$ when approaching the complement of $Amp (\tilde \chi)$, $w + A(\rho - \varphi_t)$ reaches its maximal value at some point $x_{max} \in M$. Differentiating $w + A(\rho - \varphi_t)$ twice,
\begin{equation}
	0 = \sum_k X_{k\bar ki} + A(\rho_i - \varphi_{t,i})
\end{equation}
and
\begin{equation}
\label{derivative-6-8}
	0 \geq \sum_k X_{k\bar ki\bar i} + A(\rho_{i\bar i} - \varphi_{t,i\bar i}) .
\end{equation}
Multiplying \eqref{derivative-6-8} by $S_{m-1;i}$ and summing them up,
\begin{equation}
\label{inequality-6-9}
\begin{aligned}
	0 
	&\geq \sum_{i,k} S_{m - 1;i} X_{k\bar ki\bar i} + A \sum_i S_{m - 1;i} (\rho_{i\bar i} - \varphi_{t,i\bar i}) \\
%	&= \sum_{i,k} S_{m - 1;i} X_{i\bar ik\bar k} + \sum_{i,k} S_{m - 1;i} \left(- R_{k\bar ki\bar i} X_{i\bar i} + R_{i\bar ik\bar k} X_{k\bar k} + G_{i\bar ik\bar k}\right) \\
%	&\qquad + A \sum_i S_{m - 1;i} (\rho_{i\bar i} - \varphi_{t,i\bar i}) \\
%	&= \sum_{i,k} S_{m - 1;i} X_{i\bar ik\bar k} + \sum_{i,k} S_{m - 1;i} R_{i\bar ik\bar k} \left(  X_{k\bar k}   -  X_{i\bar i} \right) + \sum_{i,k} S_{m - 1;i}  G_{i\bar ik\bar k}  \\
%	&\qquad + A \sum_i S_{m - 1;i} (\rho_{i\bar i} - \varphi_{t,i\bar i}) \\
%	&= \sum_{i,k} S_{m - 1;i} X_{i\bar ik\bar k} + \sum_{i\neq k} S_{m - 1;i} R_{i\bar ik\bar k} \left(  X_{k\bar k}   -  X_{i\bar i} \right)  \\
%	&\qquad + \sum_{i,k} S_{m - 1;i}  G_{i\bar ik\bar k}  + A \sum_i S_{m - 1;i} (\rho_{i\bar i} - \varphi_{t,i\bar i}) \\
%	&= \sum_{i,k} S_{m - 1;i} X_{i\bar ik\bar k} + \sum_{i < k} S_{m - 1;i} R_{i\bar ik\bar k} \left(  X_{k\bar k}   -  X_{i\bar i} \right)  + \sum_{i> k} S_{m - 1;i} R_{i\bar ik\bar k} \left(  X_{k\bar k}   -  X_{i\bar i} \right) \\
%	&\qquad + \sum_{i,k} S_{m - 1;i}  G_{i\bar ik\bar k}  + A \sum_i S_{m - 1;i} (\rho_{i\bar i} - \varphi_{t,i\bar i}) \\
%	&= \sum_{i,k} S_{m - 1;i} X_{i\bar ik\bar k} + \sum_{i < k} S_{m - 1;i} R_{i\bar ik\bar k} \left(  X_{k\bar k}   -  X_{i\bar i} \right)  + \sum_{k > i} S_{m - 1;k} R_{k\bar ki\bar i} \left(  X_{i\bar i}   -  X_{k\bar k} \right) \\
%	&\qquad + \sum_{i,k} S_{m - 1;i}  G_{i\bar ik\bar k}  + A \sum_i S_{m - 1;i} (\rho_{i\bar i} - \varphi_{t,i\bar i}) \\
	&= \sum_{i,k} S_{m - 1;i} X_{i\bar ik\bar k} + \sum_{i < k} \left(S_{m - 1;i} - S_{m-1;k} \right) R_{i\bar ik\bar k} \left(  X_{k\bar k}   -  X_{i\bar i} \right)  \\
	&\qquad + \sum_{i,k} S_{m - 1;i}  G_{i\bar ik\bar k}  + A \sum_i S_{m - 1;i} (\rho_{i\bar i} - \varphi_{t,i\bar i}) \\
	&\geq \Delta_\omega F_{t} - \frac{m - 1}{m} \frac{|\nabla F_t|^2}{F_t} + \sum_{i,k} S_{m - 1;i}  G_{i\bar ik\bar k}  + A \sum_i S_{m - 1;i} (\rho_{i\bar i} - \varphi_{t,i\bar i}) ,
\end{aligned}
\end{equation}
where
\begin{equation*}
	G_{i\bar ik\bar k} = \chi_{k\bar ki\bar i} + \tilde \chi_{k\bar ki\bar i} - \chi_{i\bar ik\bar k} - \tilde \chi_{i\bar ik\bar k} + \sum_k R_{k\bar ki\bar l} (\chi_{l\bar i} + \tilde \chi_{l\bar i}) - \sum_l R_{i\bar ik\bar l} (\chi_{l\bar k} + \tilde \chi_{l\bar k}) .
\end{equation*}
Substituting \eqref{inequality-6-5} into \eqref{inequality-6-9}, 
\begin{equation}
	- \Delta_\omega F_t + \frac{m - 1}{m} \frac{|\nabla F_t|^2}{F_t} + A \kappa m F_t + C (\chi, \tilde \chi, \omega) \sum_i S_{m - 1;i} \geq A \kappa  \sum_i  S_{m - 1;i}  .
\end{equation}
Choosing $A$ such that $A\kappa \geq C (\chi,\tilde \chi ,\omega) + 1$,
%\begin{equation}
%\begin{aligned}
%	- \Delta_\omega F_t + \frac{m - 1}{m} \frac{|\nabla F_t|^2}{F_t} + A \kappa m F_t   
%	&\geq   (n - m + 1) S_{m - 1} \\
%	&\geq 
%	m \left( \frac{(n - 1)!}{m! (n - m)!}\right)^{\frac{1}{m - 1}} S^{\frac{m - 2}{m - 1}}_m S^{\frac{1}{m - 1}}_1 \\
%	&=
%	m \left( \frac{(n - 1)!}{m! (n - m)!}\right)^{\frac{1}{m - 1}} F^{\frac{m - 2}{m - 1}}_t S^{\frac{1}{m - 1}}_1	
%	.
%\end{aligned}
%\end{equation}
%After simplification,
%\begin{equation}
%	- m C^m_n e^{mb_t} e^{(m - 1) f} \Delta_\omega e^f + A \kappa m C^m_n e^{mb_t} e^{mf} 
%	\geq 	
%	m \left( \frac{(n - 1)!}{m! (n - m)!}\right)^{\frac{1}{m - 1}} \left(C^m_n e^{mb_t} e^{mf}\right)^{\frac{m - 2}{m - 1}} S^{\frac{1}{m - 1}}_1	
%	. 
%\end{equation}
%\begin{equation}
%	- \left(C^m_n e^{mb_t} e^f \right)^{\frac{1}{m - 1}}  \Delta_\omega e^f + A \kappa  \left(C^m_n e^{mb_t} e^{mf} \right)^{\frac{1}{m - 1}} 
%	\geq 	
%	 \left( \frac{(n - 1)!}{m! (n - m)!}\right)^{\frac{1}{m - 1}}  S^{\frac{1}{m - 1}}_1	
%	. 
%\end{equation}
%\begin{equation}
%	- \left( e^{mb_t} e^f \right)^{\frac{1}{m - 1}}  \Delta_\omega e^f + A \kappa  \left( e^{mb_t} e^{mf} \right)^{\frac{1}{m - 1}} 
%	\geq 	
%	 \left(  \frac{1}{n}\right)^{\frac{1}{m - 1}}  S^{\frac{1}{m - 1}}_1	
%	. 
%\end{equation}
\begin{equation}
	2^{m - 2} n e^{mb_t} e^f \left( |\Delta_\omega e^f|^{m - 1} + A^{m - 1} \kappa^{m - 1} e^{(m - 1)f}\right)
	\geq 
	S_1
	. 
\end{equation}
Therefore,
\begin{equation}
\begin{aligned}
	w 
	& \leq  \left(w + A (\rho - \varphi_t)\right)\Big|_{z_{max}} - A (\rho - \varphi_t) \\
	& \leq \left(w - A  \varphi_t \right)\Big|_{z_{max}} - A  \rho \\
	& \leq C \left(\chi,\tilde \chi, \omega, \Vert e^{f}\Vert_{C^2}, - \inf_M \varphi_t\right) + A\rho .
\end{aligned}
\end{equation}

By Arzel\`a-Ascoli, the solution to Equation~\eqref{equation-1} is smooth in $Amp(\tilde \chi)$.

\begin{theorem}
Suppose that $(M,\omega)$ is a compact K\"ahler manifolds  without boundary of complex dimensional $n \geq 2$, whose bisectional curvature is nonnegative. 
Let $\chi \in \bar\Gamma^m_\omega$ be a closed real $(1,1)$-form, and $\tilde \chi$ be a semipositive and big closed $(1,1)$-form.
If $e^{f}$ is smooth,
then there is a  bounded solution in pluripotential sense to Equation~\eqref{equation-1}, which is smooth in  $Amp (\tilde\chi)$.

\end{theorem}

By Theorem~\ref{theorem-4-4}, we can prove the following continuity result. 
\begin{corollary}
Suppose that $(M,\omega)$ is a compact K\"ahler manifolds  without boundary of complex dimensional $n \geq 2$, whose bisectional curvature is nonnegative. 
Let $\chi \in \bar\Gamma^m_\omega$ be a closed real $(1,1)$-form, and $\tilde \chi$ be a semipositive and big closed $(1,1)$-form.
For $q > 1$ and $e^{nf} \in L^q$,
there is a  bounded solution in pluripotential sense to Equation~\eqref{equation-1}, which is continuous in  $Amp (\tilde\chi)$.
\end{corollary}

\medskip

\noindent
{\bf Acknowledgements}\quad
The author wish to thank Chengjian Yao, Ziyu Zhang and Long Li for their helpful discussions and suggestions. The  author is supported by a start-up grant from ShanghaiTech University.

\end{document}